\documentclass[11pt]{article}

\usepackage{amsmath}
\usepackage{amssymb}
\usepackage{amscd}
\usepackage{amsthm}
\usepackage{indentfirst}

\usepackage{calrsfs}

\usepackage{tikz}
\usetikzlibrary{matrix,decorations,arrows}
\usetikzlibrary{decorations.markings}

\usepackage{hyperref}
\hypersetup{colorlinks=true,linkcolor=red,pdfpagemode=UseNone,pdfstartview=FitH}

\newcommand{\Z}{\ensuremath{\mathbb{Z}}}
\newcommand{\Q}{\ensuremath{\mathbb{Q}}}

\newcommand{\F}{\ensuremath{\mathbb{F}}}

\usepackage[margin=2.6cm]{geometry}

\theoremstyle{plain}
\newtheorem{theorem}{Theorem}[section]
\newtheorem{lemma}[theorem]{Lemma}
\newtheorem{corollary}[theorem]{Corollary}
\newtheorem{proposition}[theorem]{Proposition}

\theoremstyle{definition}
\newtheorem{definition}[theorem]{Definition}
\newtheorem{remark}[theorem]{Remark}

\DeclareMathOperator{\Spec}{Spec}
\DeclareMathOperator{\Pic}{Pic}
\DeclareMathOperator{\ord}{ord}
\DeclareMathOperator{\divisor}{div}
\DeclareMathOperator{\Sym}{Sym}

\DeclareMathOperator{\degi}{ideg}
\DeclareMathOperator{\degs}{sdeg}

\DeclareMathOperator{\Jac}{Jac}
\DeclareMathOperator{\rk}{rk}

\DeclareMathOperator{\car}{char}
\DeclareMathOperator{\Norm}{N}

\newcommand{\PP}{\mathbb{P}}
\newcommand{\E}{\mathcal{E}}

\newcommand{\LN}{\mathrm{LN}}
\newcommand{\Tr}{\mathrm{Tr}}

\begin{document}

\title{Counting rational points on elliptic and hyperelliptic curves over function fields}

\author{Jean Gillibert \and Emmanuel Hallouin \and Aaron Levin}

\date{October 2025}

\maketitle

\begin{abstract}
Combining $2$-descent techniques with Riemann-Roch and B{\'e}zout's theorems, we give an upper bound on the number of rational points of bounded height on elliptic and hyperelliptic curves over function fields of characteristic $\neq 2$. We deduce an upper bound on the number of $S$-integral points, where $S$ is a finite set of places. As a primary application, over small finite fields we bound the $3$-torsion of Jacobians of hyperelliptic curves and the $2$-torsion of Jacobians of trigonal curves. In this setting, these bounds improve on both the trivial geometric bound and the naive inequality coming from the Weil bound, as well as recent upper bounds on $2$-torsion in the work of Bhargava \emph{et al.}.
\end{abstract}


\section{Introduction}

Let $k$ be an algebraically closed field of characteristic not $2$, and let $B$ be a smooth projective irreducible $k$-curve of genus $g$.
In this paper, we consider a hyperelliptic curve $\mathcal{C}$ over $k(B)$ defined by a Weierstrass equation of the form
\begin{equation}
\label{eq:C}
y^2 = f(x),
\end{equation}
where $f$ is a monic separable polynomial of odd degree $d\geq 3$ with coefficients in $k(B)$.

We shall use $2$-descent computations to study rational points of bounded naive height on this Weierstrass model of $\mathcal{C}$, that is, the set
\begin{align*}
\mathcal{C}(k(B))_{\leq c}:=\{(x_0,y_0)\in \mathcal{C}(k(B))\mid \deg x_0\leq c\}.
\end{align*}
It is known that this set is finite under each of the following assumptions (see \S{}\ref{sec:heights}):

\begin{enumerate}
\item[(A1)] $\mathcal{C}$ is a non-constant elliptic curve
\item[(A2)] $\car(k)=0$ and $\mathcal{C}$ is non-constant
\item[(A3)] $\car(k)>0$ and $\mathcal{C}$ is not isotrivial.
\end{enumerate}

Recall that $\mathcal{C}$ is \emph{constant} if one can obtain it by base-change from a curve defined over $k$, and that $\mathcal{C}$ is \emph{isotrivial} if it becomes constant over some finite extension of $k(B)$.

In fact, when $\mathcal{C}$ has genus at least $2$ then the whole set of rational points of $\mathcal{C}$ is known to be finite under (A2) or (A3) (this is the function field version of Mordell's conjecture).

\subsection*{Counting rational points}

Our main result is the following:

\begin{theorem}
\label{thm:intro1}
Let $\mathcal{C}$ be the hyperelliptic curve defined over $k(B)$ by the equation $y^2=f(x)$ where $f$ is monic and separable, of odd degree $d\geq 3$. Let $h_B(f)$ be the height of $f$ (see \S{}\ref{sec:heights}). Let $C_f$ be the smooth projective curve defined over $k$ by the equation $f(x)=0$, and let $\omega(f)$ be the number of irreducible components of $C_f$. If $C_f$ is irreducible, we denote its genus by $g_f$. Finally, let $J$ be the Jacobian of $\mathcal{C}$ over $k(B)$, let $\Tr_{k(B)/k}(J)$ be the $k(B)/k$-trace of $J$, and let
$$
\LN(J) := J(k(B))/\Tr_{k(B)/k}(J)(k)
$$
be the Lang-N{\'e}ron group of $J$ relative to $k(B)/k$ (which, according to the Lang-N{\'e}ron theorem, is finitely generated).
If $c$ is a positive integer, we let
$$
\Omega(c,f,g) :=
\begin{cases}
 \max\left\{d(c+g)+h_B(f)-g_f, \frac{1}{2}(d(c+g)+h_B(f))\right\} & \text{if $C_f$ is irreducible} \\
 d(c+g)+h_B(f) + \omega(f) -1 & \text{otherwise.}
\end{cases}
$$
Assume that one of the assumptions (A1), (A2) or (A3) holds. Then
\begin{enumerate}
\item[(1)] For any positive integer $c$ we have
\begin{align*}
|\mathcal{C}(k(B))_{\leq c}| \leq 2^{\Omega(c,f,g) + \max\{c, \frac{1}{2}(c+g)\} +1 + \rk_\Z \LN(J) + \dim_{\F_2} \LN(J)[2]}.
\end{align*}
\item[(2)] When the base field is $k(t)=k(\mathbb{P}^1)$, and $f$ has coefficients in $k[t]$, one can improve this as follows:
\begin{align*}
|\mathcal{C}(k(B))_{\leq c}| \leq 2^{\Omega\left(\left\lceil\frac{c}{2}\right\rceil,f,0\right)
+ \left\lceil\frac{c}{2}\right\rceil +1 + \rk_\Z \LN(J) + \dim_{\F_2} \LN(J)[2]}.
\end{align*}
\end{enumerate}
\end{theorem}

It follows from the Grothendieck-Ogg-Shafarevich formula (see \S{}\ref{sec:intro1proof}) that
\begin{equation}
\label{eq:GOSpimp}
\rk_\Z \LN(J) + \dim_{\F_2} \LN(J)[2] \leq 2d_0 +(d-1)(2g-2)+ \deg(\mathfrak{f}_J)+\omega(f) -1
\end{equation}
where $d_0$ is the dimension of $\Tr_{k(B)/k}(J)$, and $\mathfrak{f}_J$ is the conductor of $J$. This allows one to deduce from the statements above an upper bound in terms of more computable invariants of $J$.

Alternatively, one can get a bound in terms of $\mathcal{C}$ by observing that $\deg(\mathfrak{f}_J)\leq \deg(\mathfrak{f}_\mathcal{C})$ where $\mathfrak{f}_\mathcal{C}$ is the Artin conductor of $\mathcal{C}$.

Note that $\omega(f)$ is none other than the number of irreducible factors of $f$ as a polynomial over $k(B)$, and that $\dim_{\F_2} J(k(B))[2] = \omega(f)-1$.

In the case $d=3$, our curve $\mathcal{C}$ is (under the assumptions of Theorem~\ref{thm:intro1}) a non-constant elliptic curve, hence its $k(B)/k$-trace vanishes. This means that $d_0=0$ and that the Lang-N{\'e}ron group agrees with the group of $k(B)$-rational points of $\mathcal{C}$.
In this case, it is a classical result of N{\'e}ron that $|\mathcal{C}(k(B))_{\leq c}|\sim \beta\cdot c^{r/2}$ as $c$ tends to infinity, where $r$ is the rank of $\mathcal{C}(k(B))$ and $\beta$ is a constant depending on $\mathcal{C}$. Although the bounds in Theorem~\ref{thm:intro1} are asymptotically weaker in comparison, the key point is that they give an explicit estimate of the number of points of small height; the precise form of our bounds will be crucial in the applications in Section \ref{section4} to bounding the $3$-torsion of hyperelliptic Jacobians over finite fields.

The strategy of our proof is mainly geometric: it relies on a counting argument for the number of points of bounded height which map to the same class under the $2$-descent map. This amounts to counting functions in certain linear systems, the main tool being the Riemann-Roch theorem. One of its strengths is that it is characteristic-free. One weakness is its geometric nature: one cannot expect an improvement when the base field is not algebraically closed. The end of the proof is classical: we bound the size of the image of the $2$-descent map in terms of the rank and the size of the $2$-torsion subgroup. This part is sensitive to the base field, and indeed in our primary application we exploit this by taking advantage of an upper bound on the rank of an elliptic curve over $\F_q(t)$ due to Brumer.

\subsection*{Counting integral points}

We now choose a finite non-empty set $S\subset B$ and denote by $R_S\subset k(B)$ the ring of functions with no poles outside $S$. We assume that $f$ has coefficients in $R_S$.
We are now interested in the set of $S$-integral points on the given Weierstrass model of $\mathcal{C}$, namely
\begin{align*}
\mathcal{C}(R_S):=\{(x_0,y_0)\in \mathcal{C}(k(B))\mid x_0,y_0\in R_S\}.
\end{align*}

Under assumption (A2) or (A3) this set is finite. More precisely, when $\mathcal{C}$ has genus $1$ then the set of $S$-integral points is known to be finite by results of Lang \cite{lang1960} when $\car(k)=0$, and Voloch \cite{voloch90} when $\car(k)>0$; when $\mathcal{C}$ has genus at least two then the set of rational points is finite. See \S{}\ref{sec:heights} for the details.

In $\S{}\ref{sec:intpointsheightbound}$ we give an upper bound on the height of $S$-integral points on $\mathcal{C}$, following previous work of Hindry-Silverman. This allows us to derive from Theorem~\ref{thm:intro1} an upper bound on the number of $S$-integral points, which reads as follows:

\begin{theorem}
\label{thm:intro2}
With the same notation as in Theorem~\ref{thm:intro1}, assume (A2) holds, or (A3) with $\car(k)>d$. Assume that $f$ has coefficients in $R_S$, and let $\Delta_f$ be the discriminant of $f$. Let $\rho$ be the inseparable degree defined in Theorem~\ref{thm:naiveheightbound}, and let
$$
c_{\max} :=
\begin{cases}
 4(2g - 2 + |S| + \deg\Delta_f) + \frac{3h_B(f)}{d} & \text{if $\car(k)=0$} \\
 6\rho(2g - 2 + |S| + \deg\Delta_f) + \frac{3h_B(f)}{d} & \text{if $\car(k)>d$.}
\end{cases}
$$
Then we have
$$
|\mathcal{C}(R_S)| \leq
\begin{cases}
2^{(d+1)c_{\max}+h_B(f) + dg - g_f +1 + \rk_\Z \LN(J)} & \text{if $C_f$ is irreducible} \\
2^{(d+1)c_{\max}+h_B(f) + dg + \omega(f) + \rk_\Z \LN(J) + \dim_{\F_2} \LN(J)[2]} & \text{otherwise.}
\end{cases}
$$
\end{theorem}

In order to bound the height of integral points, the key ingredient is the \emph{abc}-theorem over function fields, that we apply over a splitting field of $f$. The condition $\car(k)>d$ is used to ensure that this extension is tamely ramified over $k(B)$.

In the case when the base curve is $\mathbb{P}^1$, one has an improved bound as in Theorem~\ref{thm:intro1}~(2). One can also deduce an alternative bound by combining this with the inequality \eqref{eq:GOSpimp}.

When $k$ has characteristic $0$ and $d=3$ (\emph{i.e.} $\mathcal{C}$ is an elliptic curve), Hindry and Silverman \cite{HS} proved, under the assumption that the Weierstrass equation of $\mathcal{C}$ is minimal over the ring $R_S$, that
$$
|\mathcal{C}(R_S)| \leq
\left\{
\begin{array}{ll}
144\left(10^{7.1}\sqrt{|S|}\right)^r  & \text{if}\deg(\mathcal{D})\geq 24(g-1) \\
(8\pi^2(g-1))^{2/3}\left(10^{7+12g}\sqrt{|S|}\right)^r & \text{otherwise},
\end{array}\right.
$$
where $\mathcal{D}$ is the discriminant of (the Weierstrass equation of) $\mathcal{C}$, and $r=\rk_\Z \mathcal{C}(k(B))$ is the rank of $\mathcal{C}$. This result was extended to function fields of positive characteristic by Pacheco \cite{pacheco98}.  Our bound improves on Hindry and Silverman's one only in specific ranges (e.g. if $S$ is small enough), but also applies without a minimality hypothesis. Yet another bound on the number of integral points, valid in characteristic 0 and without a minimality hypothesis, was given by Chi, Lai, and Tan \cite{CLT04}; their bound may have advantages over the Hindry-Silverman bound in certain cases when $k$ is not algebraically closed.

\subsection*{Bounding the $3$-torsion of Jacobians of hyperelliptic curves}

When the base curve is $\mathbb{P}^1$, $S=\{\infty\}$, and $\mathcal{C}$ is an elliptic curve, we are able to improve the bound for the number of integral points in $\mathcal{C}(k[t])$ with small naive height by using refinements of Riemann-Roch specific to trigonal curves, originating in classical results of Maroni (see Theorem~\ref{generalMaronibound}).

By relating $3$-torsion points of Jacobians of hyperelliptic curves to integral points on certain elliptic curves, we deduce the following result.

\begin{theorem}
\label{thm:intro3}
Let $q=p^r$ for some prime $p\geq 5$. Let $X$ be a hyperelliptic curve of genus $g$ over $\mathbb{F}_q$, with a rational Weierstrass point, and let $\Jac(X)$ be the Jacobian of $X$. Then
\begin{align*}
|\Jac(X)(\mathbb{F}_q)[3]|\leq q^{\frac{g}{2}+\gamma\frac{g}{\log g}}
\end{align*}
for some explicit constant $\gamma$ depending only on $q$.
\end{theorem}

When $q<81$ this asymptotically improves on the trivial bound $|\Jac(X)(\mathbb{F}_q)[3]|\leq 3^{2g}$. 
When $\mathbb{F}_q$ does not contain a primitive third root of unity (\emph{i.e.} when $q\not\equiv 1\pmod{3}$) then $|\Jac(X)(\mathbb{F}_q)[3]|\leq 3^{g}$ by Galois-invariance of the Weil pairing. We (asymptotically) improve on this bound when $q<9$.

Weil \cite{Weil} proved the inequalities $(\sqrt{q}-1)^{2g}\leq |\Jac(X)(\mathbb{F}_q)|\leq (\sqrt{q}+1)^{2g}$, and in particular $|\Jac(X)(\mathbb{F}_q)[3]|\leq (\sqrt{q}+1)^{2g}$. An argument of Soundararajan outlined in \cite[p.~19]{HV} (see also \cite{Yudovina}), when applied over function fields using the (known) generalized Riemann hypothesis in that setting, improves this to $|\Jac(X)(\mathbb{F}_q)[3]|\leq q^{\frac{2}{3}g+\epsilon}$ for any $\epsilon>0$. Our result (asymptotically) improves  both of these bounds.

To compare with analogous results over number fields, we note that curves $X$ of genus $g$ and gonality $n$ over $\mathbb{F}_q$ are analogous to number fields $k$ of degree $n$ over $\mathbb{Q}$, and the absolute discriminant $\Delta_k$ of $k$ is analogous to $q^{2g}$. If we write $\Delta_X=q^{2g}$, then our bound is of the form $\Delta_X^{\frac{1}{4}+\epsilon}$, and the hyperelliptic curves $X$ are analogous to quadratic fields over $\mathbb{Q}$. After work of Pierce \cite{Pierce05,Pierce06} and Helfgott and Venkatesh \cite{HV}, the best known upper bound for the size of the $3$-part of the ideal class group of a quadratic field $k$ over $\mathbb{Q}$ is $\Delta_k^{\frac{1}{3}+\epsilon}$, due to Ellenberg and Venkatesh \cite{EV}. Thus, we obtain an improvement over these results in the function field setting.

\begin{remark}
\label{remark:goodreductiontrick}
Let $X$ be a hyperelliptic curve of genus $g$ over $\Q$, with a rational Weierstrass point. If $X$ has good reduction at $5$, then
$$
|\Jac(X)(\Q)[3]|\leq 5^{\frac{g}{2}+\gamma\frac{g}{\log g}}
$$
for some absolute constant $\gamma$, which asymptotically improves on the trivial bound $3^g$. Indeed, the reduction-mod-$5$ map is injective on $3$-torsion hence the result follows from Theorem~\ref{thm:intro3}. Under the weaker assumption that $\Jac(X)$ has good reduction at $5$, one has a slightly modified variant, because in this case the reduction of $\Jac(X)$ is a product of Jacobians of hyperelliptic curves, whose sum of genera is equal to $g$ (see Remark~\ref{remark:Xbadred}).

A similar statement holds when $X$ has good reduction at $7$. If $X$ has bad reduction, one can still give an upper bound depending on the reduction type of $X$ (see Remark~\ref{remark:Xbadred}).
\end{remark}

\begin{remark}
Let $T\to \PP^1$ be a trigonal curve, whose Galois closure $\tilde{T}\to \PP^1$ has group $S_3$, and let $X\to \PP^1$ be the unique hyperelliptic subcover of $\tilde{T}$. Spencer \cite{spencer2025} constructs a Galois-equivariant map $\Jac(T)[3] \to \Jac(X)[3]$ which is injective when $g(X) = g(T)$ or $g(X) = g(T) + 1$. In these cases, one can derive from Theorem~\ref{thm:intro3} an upper bound on $\Jac(T)(\mathbb{F}_q)[3]$.
\end{remark}

Finally let us make a small comment on the case when $\car(k)=2$. In order to extend our results to this case, one should replace {\'e}tale cohomology by flat cohomology. The $2$-descent map can be still described explicitly, but the formulas are more involved. Let us cite the results of Kramer \cite{kramer77} who worked out the case of an ordinary elliptic curve over a field of characteristic $2$. In this case multiplication-by-$2$ can be decomposed into Frobenius and Verschiebung isogenies, and the $2$-descent mixes Kummer theory and Artin-Schreier theory. This is beyond the scope of the current paper.

\subsection*{Bounding the $2$-torsion of Jacobians of trigonal curves}

Using techniques similar to those described in the previous section, by relating $2$-torsion points of Jacobians of trigonal curves to integral points on certain elliptic curves, we deduce the following result.

\begin{theorem}
\label{thm:intro4}
Let $q=p^r$ for some prime $p\geq 5$. Let $\pi:X\to\mathbb{P}^1$ be a trigonal curve of genus $g$ over $\mathbb{F}_q$, with a rational totally ramified point, and let $\Jac(X)$ be the Jacobian of $X$. Then
\begin{align*}
|\Jac(X)(\mathbb{F}_q)[2]|\leq (2q)^{\frac{g}{3}+\gamma\frac{g}{\log g}},
\end{align*}
for some explicit constant $\gamma$ depending only on $q$.
\end{theorem}

Bhargava \emph{et al.} \cite[Theorem~7.1]{bhargava17} proved the general upper bound (without a trigonal hypothesis)
\begin{align*}
|\Jac(X)(\mathbb{F}_q)[2]|\leq \frac{q^{g+1}-1}{q-1},
\end{align*}
and in the case of $n$-gonal curves the bound
\begin{align*}
|\Jac(X)(\mathbb{F}_q)[2]|\ll_n q^{(1-\frac{1}{n})g}.
\end{align*}

It is worth noting that their proof relies on the Riemann-Roch Theorem, like ours.

In the case of trigonal curves ($n=3$) with a rational totally ramified point and $p\geq 5$, our Theorem~\ref{thm:intro4} asymptotically improves on these bounds for all valid values of $q$. When $q<32$, 
we asymptotically improve on the trivial bound $|\Jac(X)(\mathbb{F}_q)[2]|\leq 2^{2g}$.

\begin{remark}
Let $E$ be a non-constant elliptic curve over $\F_q(t)$. Assume that $E$ has at least one rational place of additive reduction of type $\mathrm{II}$, $\mathrm{IV}$, $\mathrm{II}^*$ or $\mathrm{IV}^*$ (hence the trigonal curve over $\F_q$ defined by the vanishing of the $y$-coordinate on $E$ has a rational totally ramified point).
When the conductor of $E$ has large degree but only a few irreducible components, one obtains by combining \cite[Theorem~1.1]{gl18a} and Theorem~\ref{thm:intro4}  an upper bound on the rank of $E$ over $\F_q(t)$ which asymptotically improves on the geometric rank bound (see the introduction of \cite{gl18a} for relevant terminology). Under suitable assumptions, a reduction trick as in Remark~\ref{remark:goodreductiontrick} allows to replace $\F_q$ by a number field.
\end{remark}

\subsection*{Structure of the paper}

In Section~\ref{section2} we recall properties of heights and the explicit formula for the $2$-descent map, and then we prove the main result, Theorem~\ref{thm:nbratpoint}, which gives an upper bound on the number of rational points of bounded height mapping to a given class under the $2$-descent map; we then derive Theorem~\ref{thm:intro1}. In Section~\ref{section3} we give an upper bound on the height of $S$-integral points on $\mathcal{C}$, then we prove Theorem~\ref{thm:intro2}. Then we focus on elliptic curves over $\PP^1$. The refinements of Riemann-Roch for the trigonal curve $C_f$ lead to improvements on the counting of points of small height. At the core of our ingredients is the notion of Maroni invariant of a trigonal curve. In Section~\ref{section4}, we take advantage of these refinements to prove Theorem~\ref{thm:intro3} and Theorem~\ref{thm:intro4}.


\section{Counting rational points}
\label{section2}

As in the introduction, $k$ is an algebraically closed field of characteristic not $2$, and $B$ is a smooth projective geometrically connected $k$-curve of genus $g$. We consider a hyperelliptic curve $\mathcal{C}$ over $k(B)$ defined by a Weierstrass equation of the form \eqref{eq:C}.
We let $C_f$ be the smooth projective $k$-curve defined by the equation $f(x)=0$, and we denote by $\pi:C_f\to B$ the natural degree $d$ map. We let $\omega(f)$ be the number of irreducible factors of $f$ over $k(B)$, which is equal to the number of irreducible components of $C_f$. We denote by $k(C_f) = k(B)[X]/f(X)$ the ring of rational functions on $C_f$, which is a $k(B)$-algebra of degree $d$. If $C_1,\dots,C_{\omega(f)}$ are the irreducible components of $C_f$, then we have a splitting $k(C_f) = k(C_1)\times\dots \times k(C_{\omega(f)})$ where the $k(C_i)$ are fields.


\subsection{Heights}
\label{sec:heights}

The degree of a non-constant rational function on $B$ is by definition the degree of the divisor of its poles, equivalently the degree of the corresponding map $B\to \PP^1$. By convention, the degree of a constant map is zero.

Recall that we have the properties
$$
\deg(u^r)=r\deg(u); \qquad \deg(uv) \leq \deg(u) + \deg(v); \qquad \deg(u+v) \leq \deg(u) + \deg(v).
$$

If $P=(x_0,y_0)$ is a $k(B)$-rational point on $\mathcal{C}$, we consider $\deg(x_0)$ as being its naive height (as does Silverman in \cite[Chap.~III, \S{}4]{silvermanII}). In order to keep the notation as simple as possible, we shall refer to the degree in all statements, avoiding the language of heights.

Let us point out that this naive height depends on the choice of an equation for $\mathcal{C}$. Once an equation is fixed, the naive height is closely related to the N{\'e}ron-Tate height (on the Jacobian of $\mathcal{C}$); more precisely, the difference between $\frac{1}{2}\deg(x_0)$ and the N{\'e}ron-Tate height of the divisor class of $(x_0,y_0)-\infty$ is bounded by an absolute constant depending only on the Weierstrass equation of $\mathcal{C}$. In the case of elliptic curves, this is proved in \cite[Chap.~III, \S{}4]{silvermanII}. The hyperelliptic case is similar.
We give an explicit inequality in the simplest case of an elliptic curve over $k(t)$; the result is implicit in the literature, but lacking a direct reference we provide a proof.

\begin{theorem}
\label{comparison}
Let $E:y^2=x^3+Ax+B$ be a nonconstant elliptic curve over $k(t)$ with $A,B\in k[t]$. Let $\chi=\max\{\lceil\frac{1}{4}\deg A\rceil, \lceil\frac{1}{6}\deg B\rceil\}$ and let $j$ be the $j$-invariant of $E$. For $P=(x_0,y_0)\in E(k(t))$, we have
\begin{align*}
-2\chi\leq \deg(x_0)-2\hat{h}(P)&\leq \frac{1}{12}\deg(j)+2\chi.
\end{align*}
\end{theorem}

\begin{proof}
We may identify the set of places of $k(t)$ with the set of maximal ideals of $k[t]$ along with a unique place $\infty$, where $v(f/g)=\deg g-\deg f$ when $v=\infty$ and $f,g\in k[t]\setminus\{0\}$ (identifying a place with its associated discrete valuation).  For every place $v\neq\infty$, since $A$ and $B$ are polynomials (and so $v$-integral) it follows from a result of Tate (see \cite[Theorem~4.5]{Lang78} and \cite[Theorem~4.1]{Sil90}) that
\begin{align}
\label{Tate}
-\frac{1}{6}v(\Delta)\leq \max\{0,-v(x(P))\}-2\lambda_v(P)\leq \frac{1}{12}\max\{0, -v(j)\},
\end{align}
where $\hat{h}(P)=\sum_v \lambda_v(P)$ and $\lambda_v(P)$ does not depend on the choice of Weierstrass equation \cite[p.~64]{Lang78}. Now consider $v=\infty$. We change coordinates so that the coefficients in the Weierstrass equation are $v$-integral and work with the point $P'=(x',y')=(x(P)/t^{2\chi},y/t^{3\chi})$ on the curve $E':y'^2=x'^3+A/t^{4\chi}x'+B/t^{6\chi}$, with discriminant $\Delta'$ and $j$-invariant $j'$. Then applying \eqref{Tate} to $P'$ and $E'$, and using $\lambda_v(P')=\lambda_v(P)$, $j'=j$, $v(x'(P'))=v(x(P))+2\chi$, and $v(\Delta')=v(\Delta)+12\chi$, we obtain
\begin{align*}
-\frac{1}{6}v(\Delta)-2\chi\leq \max\{0,-v(x(P))-2\chi\}-2\lambda_v(P)\leq \frac{1}{12}\max\{0, -v(j)\},
\end{align*}
which implies
\begin{align*}
-\frac{1}{6}v(\Delta)-2\chi\leq \max\{0,-v(x(P))\}-2\lambda_v(P)\leq \frac{1}{12}\max\{0, -v(j)\}+2\chi.
\end{align*}
Now summing over all places yields the inequality. 
\end{proof}

Under our running assumptions, the naive height satisfies the Northcott property, \emph{i.e.} there are only finitely many rational points of bounded height on $\mathcal{C}$.

\begin{proposition}[Northcott]
Assume (A1), (A2) or (A3) holds. Then for any $c>0$ the set
$$
\mathcal{C}(k(B))_{\leq c} := \{(x_0,y_0)\in \mathcal{C}(k(B)) \mid \deg(x_0)\leq c\}
$$
is finite.
\end{proposition}

\begin{proof}
Under assumption (A1), $\mathcal{C}$ is a non-constant elliptic curve hence, according to the Lang-N\'eron Theorem, the group $\mathcal{C}(k(B))$ is finitely generated. Actually, it is part of the proof of the Lang-N\'eron Theorem that the naive height satisfies the Northcott property. For a modern exposition, we refer the reader to Conrad \cite[Section~7]{conrad06}.

When $\mathcal{C}$ has genus at least $2$ then the whole set of rational points of $\mathcal{C}$ is known to be finite under (A2) or (A3), by results of Manin \cite{manin63} and Grauert \cite{grauert65} in the case when $\car(k)=0$, completed by Samuel \cite{samuel66} when $\car(k)>0$. 
\end{proof}

We shall also use the notion of height of a polynomial with coefficients in $k(B)$. More precisely, if
$$
f = X^d + a_{d-1}X^{d-1} + \dots + a_1X + a_0,
$$
then the height of the polynomial $f$ (with respect to $B$) is defined as usual by
$$
h_B(f) := -\sum_{v\in B} \min\{0,v(a_0),\dots,v(a_{d-1})\}.
$$
We refer to \cite[Chap.~I, \S{}2]{mason84} for basic properties of this height. In particular:
\begin{enumerate}
  \item $h_B(X+a_0)=\deg(a_0)$
  \item If $f$ and $g$ are monic polynomials, then $h_B(fg)=h_B(f)+h_B(g)$
  \item When changing the base curve, the height is multiplied by the degree of the corresponding function field extension.
\end{enumerate}

It follows by combining the three previous properties that, if $B'/B$ is a cover of curves over which the monic polynomial $f$ has all its roots $e_1,\dots, e_d$, then we have
\begin{equation}
\label{eq:deg-height1}
\sum_{i=1}^d \deg_{B'}(e_i) = h_{B'}(f) = [B':B]h_B(f).
\end{equation}

Going back to our construction of the curve $C_f$, one can deduce that
\begin{equation}
\label{eq:deg-height2}
\deg_{C_f}(x) = h_B(f),
\end{equation}
where the degree of $x$ is computed over the curve $C_f$. Let us sketch the proof: since the degree of a function is additive over irreducible components of $C_f$, we may assume that $f$ is irreducible, in which case its roots are all conjugates (in the splitting field $k(B')$) hence all have the same degree. Therefore, \eqref{eq:deg-height1} yields
$$
d\deg_{B'}(x) = h_{B'}(f) = [B':B]h_B(f),
$$
which can be rewritten as
$$
d[B':C_f]\deg_{C_f}(x) = d[B':C_f]h_B(f),
$$
hence the result.


\subsection{The $2$-descent map}

Before we define this map, let us introduce some notation. Given a $k$-curve $C$ and a divisor $D$ on $C$, we identify the {\'e}tale cohomology group $H^1(C\setminus D,\mu_2)$ as a subgroup of $k(C)^\times/(k(C)^\times)^2$ as follows:
$$
H^1(C\setminus D,\mu_2)=\{h\in k(C)^\times/(k(C)^\times)^2 \mid \forall v\in C\setminus D,\quad v(h)\equiv 0\pmod{2}\}.
$$
Since $k$ is algebraically closed, this group is finite; if $C$ is irreducible then
$$
\dim_{\F_2} H^1(C\setminus D,\mu_2) =
\begin{cases}
  2g(C) &\text{if~$D=\emptyset$}\\
  2g(C)+\#D-1 &\text{if~$D\neq \emptyset$}.
\end{cases}
$$

\begin{lemma}
\label{lem:descentmap}
Let $\Sigma_2\subset B$ be the set of places above which the order of the group of connected components of the N\'eron model of the Jacobian of $\mathcal{C}$ is even, and let $x$ be the function on $C_f$ defined by $x:=X\mod{f(X)}$. Then the map
\begin{align*}
\delta: \mathcal{C}(k(B)) &\longrightarrow H^1(C_f\setminus \pi^{-1}(\Sigma_2),\mu_2) \\
(x_0,y_0) &\longmapsto x_0-x & \text{when $f(x_0)\neq 0$}\\
(x_0,0) &\longmapsto (x_0-X) + \frac{f(X)}{(X-x_0)} \mod{f(X)} & \text{when $f(x_0)=0$}
\end{align*}
is well-defined. If $\mathcal{C}$ is an elliptic curve, this map is a group homomorphism.
\end{lemma}

\begin{proof}
If $\mathcal{C}$ is an elliptic curve, then this map $\delta$ is obtained by composing the classical $2$-descent map deduced from the Kummer exact sequence with the map induced (on the $H^1$) by the Weil pairing with the generic $2$-torsion point. In general, $\delta$ is the composition of the analogous map on the Jacobian of $\mathcal{C}$ with the embedding of $\mathcal{C}$ into its Jacobian relative to the point at infinity. One can work this out from the description given by Schaefer \cite[Theorem~1.2]{schaefer95}, including the case when $f(x_0)=0$ \cite[Lemma~2.2]{schaefer95}.

Finally, the argument proving that $\delta$ has values in the group $H^1(C_f\setminus \pi^{-1}(\Sigma_2),\mu_2)$ is the same as in the proof of Proposition~4.1 of \cite{GHL23}.
\end{proof}


\subsection{Upper bound for rational points of given height}

As in the introduction, we denote by $\mathcal{C}(k(B))_{\leq c}$ the set of $k(B)$-rational points $(x_0,y_0)$ with $\deg(x_0)\leq c$. The main result of this section is the following.

\begin{theorem}
\label{thm:nbratpoint}
Let $\mathcal{C}$ be the hyperelliptic curve over $k(B)$ defined by the equation $y^2=f(x)$, where $f \in k(B)[x]$ is a monic separable polynomial of odd degree $d\geq 3$. Let $h_B(f)$ be the height of $f$, let $C_f$ be the curve defined over $k$ by the equation $f(x)=0$, and let $\omega(f)$ be the number of irreducible components of $C_f$. If $C_f$ is irreducible, we denote its genus by $g_f$.

If $c$ is a positive integer, we let
$$
\Omega(c,f,g) :=
\begin{cases}
 \max\left\{d(c+g)+h_B(f)-g_f, \frac{1}{2}(d(c+g)+h_B(f))\right\} & \text{if $C_f$ is irreducible} \\
 d(c+g)+h_B(f) + \omega(f) -1 & \text{otherwise.}
\end{cases}
$$
Assume that one of (A1), (A2) or (A3) holds. Then there are at most
$$
2^{\Omega(c,f,g) + \max\{c, \frac{1}{2}(c+g)\} +1}
$$
points in the set $\mathcal{C}(k(B))_{\leq c}$ mapping (under $\delta$) to the same class in $H^1(C_f\setminus \pi^{-1}(\Sigma_2),\mu_2)$, where $\Sigma_2$ and $\delta$ are defined in Lemma~\ref{lem:descentmap}.
\end{theorem}

Before proving the theorem, we give two preliminary lemmas.

\begin{lemma}
\label{x1lem}
Let $(x_0,y_0) \in \mathcal{C}(k(B))$ be a rational point with $y_0\neq 0$,
and let $x_1\in k(B)$ be such that $x_1-x$ and $x_0-x$ define the same class in $k(C_f)^\times/(k(C_f)^\times)^2$. Then there exists $y_1\in k(B)$ such that $(x_1,y_1)$ belongs to $\mathcal{C}(k(B))$. In particular, assuming that one of (A1), (A2) or (A3) holds, there are only finitely many $x_1\in k(B)$ with $\deg x_1 \leq c$ such that $x_1-x$ and $x_0-x$ define the same class in $k(C_f)^\times/(k(C_f)^\times)^2$.
\end{lemma}

\begin{proof}
Let us consider the norm map $\Norm_{C_f/B}:k(C_f)^\times\to k(B)^\times$. By definition, $\Norm_{C_f/B}(x_0-x)$ is the determinant of the $d\times d$ matrix (with coefficients in $k(B)$) corresponding to multiplication by $x_0-x$ in the basis $\{1,x,\dots,x^{d-1}\}$ of $k(C_f)/k(B)$. This is equal to the value at $x_0$ of the minimal polynomial of $x$, in other terms
\begin{equation}
\label{eq:normCf}
\Norm_{C_f/B}(x_0-x) = f(x_0).
\end{equation}

Now, let $x_1\in k(B)$ be such that $x_1-x$ and $x_0-x$ define the same class in $k(C_f)^\times/(k(C_f)^\times)^2$. This means that there exists a rational function $\phi$ on $C_f$ such that $x_1-x = (x_0-x)\phi^2$. In particular, we have the relation
$$
\Norm_{C_f/B}(x_1-x) = \Norm_{C_f/B}(x_0-x)\Norm_{C_f/B}(\phi)^2
$$
which, according to \eqref{eq:normCf}, can be written as
$$
f(x_1)=f(x_0)\Norm_{C_f/B}(\phi)^2.
$$

Since $(x_0,y_0)$ is a point on $\mathcal{C}$, we have $f(x_0)=y_0^2$, hence it follows from the above relation that $f(x_1)=(y_0\Norm_{C_f/B}(\phi))^2$, which proves the first claim, letting $y_1:=y_0\Norm_{C_f/B}(\phi)$.

The conclusion follows from Northcott's property (see \S{}\ref{sec:heights}): under (A1), (A2) or (A3) there are only finitely many rational points $(x_1,y_1)$ on $\mathcal{C}$ with $\deg x_1 \leq c$.
\end{proof}

Recall that for a function $z$ we let $\divisor(z)= (z)_0 - (z)_\infty$.
In the following Lemma, we prove that every function on the curve~$B$ can be written as the quotient of two functions with poles concentrated on a given
point.

\begin{lemma}
\label{lem:x0denominator}
Let $x_0\in k(B)$ of degree $\leq c$, and let $P_0$ be a closed point of $B$. 
Then there exist two functions $u_0, v_0 \in L_B((c+g)P_0)$ such that~$x_0=u_0/v_0$. In particular:
\begin{align*}
& (v_0)_\infty \leq (c+g)P_0
&
&\text{and}
&
& (v_0x_0)_\infty = (u_0)_\infty \leq (c+g)P_0.
\end{align*}
\end{lemma}

\begin{proof}
Since $-(x_0)_\infty +(c+g)P_0$ has degree $\geq g$, by Riemann-Roch there exists a function $v_0 \in k(B)^\times$ such that
$$
\divisor(v_0) - (x_0)_\infty +(c+g)P_0 \geq 0.
$$
It follows that $\divisor(v_0x_0) +(c+g)P_0 \geq 0$ and that $\divisor(v_0) +(c+g)P_0 \geq 0$ (as divisors on $B$).
\end{proof}

\begin{proof}[Proof of Theorem~\ref{thm:nbratpoint}]
In coherence with Lemma~\ref{lem:descentmap}, we let $D_2:=\pi^{-1}(\Sigma_2)$. Let us pick an arbitrary closed point $P_0$ of $B$, 
and let $D_0:=\pi^*(P_0)$.

Let $(x_0,y_0)\in \mathcal{C}(k(B))$ be a rational point on $\mathcal{C}$, such that $\deg x_0\leq c$ and $y_0\neq 0$.
Then according to Lemma~\ref{lem:x0denominator} there exist two functions $u_0,v_0 \in L((c+g)P_0)$ such that~$x_0=u_0/v_0$. On~$C_f$,
this leads to
$$
(v_0^2 x_0)_\infty \leq (v_0)_\infty + (u_0)_\infty \leq 2(c+g)D_0,
$$
and to
$$
(v_0^2 x)_\infty \leq 2(v_0)_\infty + (x)_\infty \leq 2(c+g)D_0 + (x)_\infty.
$$
Since the order of a pole of a sum is bounded above by the maximum of the order of the poles of each term,
we deduce that
$$
(v_0^2x_0-v_0^2 x)_\infty \leq 2(c+g)D_0 + (x)_\infty,
$$
or equivalently that
\begin{align}
\label{eq:v0squared}
\divisor(v_0^2x_0-v_0^2x) \geq -2(c+g)D_0 - (x)_\infty.
\end{align}

We let
$$
D_\infty:= \sum_{P~\text{pole of}~x} -\lfloor\ord_P(x)/2\rfloor.P
$$
By construction, $D_\infty$ is an effective divisor whose support is the same as $\divisor(x)_\infty$, and
\begin{equation}
\label{eq:D_x_infty}
D_\infty \leq (x)_\infty\leq 2D_\infty
\end{equation}

Let $\psi_0 :=v_0^2x_0-v_0^2x$. Then according to Lemma~\ref{lem:descentmap}, $\psi_0$ defines a class in $H^1(C_f\setminus D_2,\mu_2)$, which means that $v(\psi_0) \equiv 0 \pmod{2}$ for all $v\not\in D_2$. Summing all this up we have
\begin{equation}
\label{eq:divpsi0}
\divisor(\psi_0)=2E_0+\sum_{i=1}^s P_i -2(c+g)D_0 - 2D_\infty,
\end{equation}
where $E_0$ is an effective divisor and the points $P_i$ belong to the support of $D_2$. If a point $P_i$ appears with multiplicity two or more, one sends it inside $E_0$. If a point $P_i$ is a pole of odd order of $x$, it also appears in $D_\infty$. Then we note that $E_0$, $D_0$, $D_\infty$ and $P_i$ may have points in common.

Let $x_1$ be a rational function on $B$ with $\deg x_1\leq c$ such that $x_1-x$ and $x_0-x$ define the same class in $k(C_f)^\times/(k(C_f)^\times)^2$. As we did previously for $\psi_0$, let $\psi_1=v_1^2x_1-v_1^2x$. Then we may again write
$$
\divisor(\psi_1)=2E_1+\sum_{j=1}^l Q_j -2(c+g)D_0 - 2D_\infty,
$$
for some effective divisor $E_1$ and some points $Q_j$ in the support of $D_2$. Since $\psi_0$ and $\psi_1$ define the same class in a $2$-torsion group, we find that $\psi_0\psi_1$ is the trivial class, so is equal to $\phi^2$ for some rational function $\phi$.  Since the divisor of $\phi^2$ has even coefficients, we deduce that $\sum_{j=1}^l Q_j + \sum_{i=1}^s P_i$ has even coefficients, hence these two divisors agree. Therefore,
\begin{align*}
\divisor(\phi)=E_0+E_1+\sum_{i=1}^s P_i -2(c+g)D_0 - 2D_\infty.
\end{align*}
In particular, $\phi\in L(2(c+g)D_0 + 2D_\infty -E_0-\sum_{i=1}^sP_i)$.  Let us denote by~$V$ this space of functions and let~$\phi_0,\ldots, \phi_n$ be a basis for~$V$, where $n=h^0(2(c+g)D_0 + 2D_\infty -E_0-\sum_{i=1}^sP_i)-1$.
Then we may write $\phi=a_0\phi_0+\cdots +a_n\phi_n$ for some $a_i\in k$.  It follows that
\begin{align*}
[\phi^2]=[\psi_0\psi_1]=[(a_0\phi_0+\cdots +a_n\phi_n)^2],
\end{align*}
or
\begin{align*}
\left[\frac{\phi^2}{\psi_0}\right] =
[\psi_1]&=\left[\frac{1}{\psi_0}(a_0\phi_0+\cdots +a_n\phi_n)^2\right]\\
&=\left[\sum_{i,j}a_ia_j\frac{\phi_i\phi_j}{\psi_0}\right]
\end{align*}
(brackets mean projectively, \emph{i.e.} up to a non-zero constant of~$k$). The last equality takes place
inside~$\mathbb{P}\left(L(2(c+g)D_0 + 2D_\infty)\right)$,
and the map~$[\phi] \mapsto \left[\frac{\phi^2}{\psi_0}\right]$ is part of the following diagram:
$$
\begin{tikzpicture}[>=latex]
\matrix[matrix of math nodes,row sep=0.6cm,column sep=1cm]
{
                                                        & |(Lsym)| \mathbb{P}\left(\Sym^2(V)\right) &[-2cm]\\
\\
|(L)| \mathbb{P}\left(V\right) & |(W)| W & \subset\mathbb{P}\left(L(2(c+g)D_0 + 2D_\infty)\right) \\[-0.5cm]
|(phi)| [\phi] & |(phi2)| \left[\frac{\phi^2}{\psi_0}\right]\\
};
\draw[->, bend left] (L) to node[above,midway,anchor=east] {\text{$2$-uple Veronese}} (Lsym.west) ;
\draw[->] (Lsym.south) to node[right,midway,anchor=west] {\text{linear projection}} (W.north) ;
\draw[->] (L.east) to (W.west) ;
\draw[|->] (phi.east) to (phi2.west) ;
\end{tikzpicture}
$$
The linear projection takes into account the fact that the functions~$\frac{\phi_i\phi_j}{\psi_0}$ may not be linearly independent.
In any case, since the image of the 2-uple Veronese is known to be of degree~$2^n$ and since the degree can only decrease under a linear projection, 
the image~$W$ of the bottom map has degree bounded by~$2^n$.
By construction, elements of $W$ are functions which define the same class as $\psi_0$ in $H^1(C_f\setminus D_2,\mu_2)$.

On the other hand, we denote by $L_B((c+g)P_0)$ the Riemann-Roch space computed on $B$ (unadorned linear spaces being computed on $C_f$), and
inside~$\mathbb{P}\left(L(2(c+g)D_0 + 2D_\infty)\right)$ we consider the image $W'$ of the map
\begin{align*}
\mathbb{P}^{2,1}(L_B(2(c+g)P_0)\times L_B((c+g)P_0)) &\rightarrow \mathbb{P}\left(L(2(c+g)D_0 + 2D_\infty)\right) \\
[u_1,v_1] &\mapsto [u_1-v_1^2x].
\end{align*}
Here, $\mathbb{P}^{2,1}(L_1\times L_2)$ is the weighted projective space obtained by modding out the vector space $L_1\times L_2$ by the equivalence relation $(\lambda^2u,\lambda v)\sim (u,v)$.
By the same argument as above, involving a $2$-uple Veronese on the second linear factor, the subvariety $W'$ has degree bounded by $2^{n'}$, where $n'=h^0((c+g)P_0)-1$. Note that $W'$, unlike $W$, does not depend on $x_0$; by construction, elements of $W'$ are all functions (up to a multiplicative constant) of the form $u_1-v_1^2x$ where $u_1$ and $v_1$ are chosen in suitable Riemann-Roch spaces.

Consider the subvariety~$W_0'$ of $W'$ where $v_1\neq 0$, and the subvariety $W_0'\cap W$. We define a map
\begin{align*}
\alpha: (W_0'\cap W)(k) &\rightarrow k(C_f)^\times \\
[u_1-v_1^2x] &\mapsto \frac{u_1}{v_1^2} - x.
\end{align*}

By the construction of $W$ and $W'$, and by virtue of Lemma~\ref{lem:x0denominator}, the image of $\alpha$ contains the set of functions $x_1-x$ with $\deg x_1\leq c$ which define the same class as $x_0-x$ in $k(C_f)^\times/(k(C_f)^\times)^2$. According to Lemma~\ref{x1lem} this set is finite; let us call $m$ its cardinality.

We note that the image of $\alpha$ may be larger than the desired set, but is in any case finite, since $\deg(u_1)\leq 2(c+g)$ and $\deg(v_1^2)\leq 2(c+g)$ implies that $\deg(\frac{u_1}{v_1^2})\leq 4(c+g)$ (in fact, the functions $u_1$ and $v_1^2$ having a unique pole at the same point $P_0$, this can be reduced to $2(c+g)$), and we apply Lemma~\ref{x1lem} again. So the image of $\alpha$ is a finite set of functions, say $(x_0-x),(x_1-x),\ldots, (x_{M-1}-x)$, with $M\geq m$.

We claim that for all $i\in\{0,\dots,M-1\}$, $\alpha^{-1}(x_i-x)$ is the set of closed points of a Zariski closed subset of $W_0'\cap W$. Indeed, consider the space of functions 
\begin{align*}
V_i&=\{v(x_i-x)\mid v\in k(B), v(x_i-x)\in   L_B(2(c+g)P_0) + L_B(2(c+g)P_0)x\}\\
&=\{v(x_i-x)\mid v\in L_B(2(c+g)P_0), vx_i\in L_B(2(c+g)P_0)\}.
\end{align*}
From the latter description, $V_i$ is clearly a linear subspace of $L(2(c+g)D_0 + 2D_\infty)$. From the former description, $\alpha^{-1}(x_i-x)=\mathbb{P}(V_i)\cap (W_0'\cap W)(k)$, and the claim follows. Let $W_i=\mathbb{P}(V_i)\cap (W_0'\cap W)$. Then $W_0'\cap W=\cup_{i=0}^{M-1} W_i$, and from the definitions $W_i\cap W_j=\emptyset$ if $i\neq j$. It follows that $W'\cap W$ must have at least $M\geq m$ irreducible components. On the other hand, since $\deg W\leq 2^n$ and $\deg W'\leq 2^{n'}$, by a suitable version of B{\'e}zout's theorem \cite[Example 8.4.6]{Fulton98}, $W\cap W'$ has at most $2^{n+n'}$ irreducible components. Therefore $m\leq 2^{n+n'}$, and it follows that the number of points in ${\mathcal{C}}(k(B))_{\leq c}$ having the same image by the $2$-descent map is bounded above by $2^{n+n'+1}$ (we pick up an extra factor of $2$ since there are two rational points $(x_1,y_1)\in \mathcal{C}(k(B))$ corresponding to each $x_1$).

Finally, we note that $2(c+g)D_0 + 2D_\infty -\sum_{i=1}^sP_i\sim 2E_0$ by \eqref{eq:divpsi0}, and since $\deg(\sum_{i=1}^sP_i)\geq 0$ it follows that
\begin{align*}
\deg E_0 &\leq \deg\left((c+g)D_0 + D_\infty\right) \\
        &\leq d(c+g)+\deg x & \text{by \eqref{eq:D_x_infty}} \\
        &\leq d(c+g)+ h_B(f) & \text{by \eqref{eq:deg-height2}}
\end{align*}
where $h_B(f)$ denotes the height of the polynomial $f$, computed on the base curve $B$. We also derive from \eqref{eq:divpsi0} that $2(c+g)D_0 + 2D_\infty -E_0-\sum_{i=1}^sP_i\sim E_0$, hence
$$
n=h^0(2(c+g)D_0 + 2D_\infty -E_0-\sum_{i=1}^sP_i)-1=h^0(E_0)-1
$$
Assuming first that $C_f$ is irreducible, we have
\begin{align*}
n=h^0(E_0)-1 &\leq \max\{\deg E_0-g_f, \frac{1}{2}\deg E_0\}\\
  &\leq \max\left\{d(c+g)+h_B(f)-g_f, \frac{1}{2}(d(c+g)+h_B(f))\right\}
\end{align*}
by Riemann-Roch and Clifford's theorem. Similarly,
\begin{align*}
n'&=h^0((c+g)P_0))-1 \\
&\leq \max\{c, \frac{1}{2}(c+g)\}
\end{align*}
and the result follows. If $C_f$ is not irreducible, then the upper bound on $n$ no longer holds. However, since the divisor $E_0$ is effective, we have \cite[\S{}7.3.2, Prop.~3.25]{liu}
$$
h^0(E_0) \leq \deg E_0 + \dim_k H^0(C_f,\mathcal{O}_{C_f}) \leq \deg E_0 + \omega(f)
$$
which yields, by the same reasoning as above, an upper bound on the number of points which are not of the form ($x_0,0)$ and map to the same class.

We now consider points of the form $(x_0,0) \in \mathcal{C}(k(B))$, if any. The strategy is the following: if $f$ has a root $x_0$ then, by considering a modified $2$-descent map, one can slightly improve on the previous bound for all $c>0$, and the difference between the improved bound and the one in the statement is larger than $\omega(f)$, which is an obvious upper bound on the number of points of the form $(x_0,0)$. This argument does not requires us to consider the image of points $(x_0,0)$ by the $2$-descent map $\delta$.

Let us construct this modified $2$-descent map. Assume that $x_0\in k(B)$ is a root of $f$, and let us write $f(X)=(X-x_0)\varphi(X)$ for some monic polynomial $\varphi$. Then  $C_f=B \cup C_\varphi$ (disjoint union of smooth curves), where $C_\varphi$ is the $k$-curve defined by $\varphi=0$. We have
\begin{equation}
\label{eq:coh_splitting}
H^1(C_f\setminus D_2,\mu_2) = H^1(B\setminus \Sigma_2,\mu_2) \oplus H^1(C_\varphi\setminus D_2',\mu_2)
\end{equation}
where $D_2'$ is the restriction of $D_2$ to $C_\varphi$. We observe (Lemma~\ref{x1lem}) that the $2$-descent map $\delta$ takes its values in the kernel of the norm map $H^1(C_f\setminus D_2,\mu_2) \to H^1(B\setminus \Sigma_2,\mu_2)$. If we represent a class in $H^1(C_f\setminus D_2,\mu_2)$ as a couple $(\mu,\nu)\in k(B)^\times \times k(C_\varphi)^\times$ (modulo squares), then the norm of this class is represented by $\mu\cdot N_\varphi(\nu)$ (modulo squares) where $N_\varphi$ it the norm relative to $k(C_\varphi)/k(B)$. It follows that projecting on the second factor in \eqref{eq:coh_splitting} yields an isomorphism between the kernel of the norm map and $H^1(C_\varphi\setminus D_2',\mu_2)$, the inverse map being given by $\nu\mapsto (N_\varphi(\nu),\nu)$. Therefore, by composing $\delta$ with the projection onto the second factor of \eqref{eq:coh_splitting} we obtain a modified $2$-descent map $\mathcal{C}(k(B))\to H^1(C_\varphi\setminus D_2',\mu_2)$, defined by the same formula and having the same properties as the original one. The same counting argument applies to this modified $2$-descent map, with the small improvement that there is one component less on the curve $C_\varphi$, so that $\omega(f)$ is replaced by $\omega(f)-1$. The resulting bound is half the size of the previous bound, hence the result.
\end{proof}

When the base curve $B$ is the projective line, one can improve on Theorem~\ref{thm:nbratpoint} under the additional assumption that $f$ has coefficients in $k[t]$ (which can be achieved after a suitable change of coordinates).

\begin{proposition}
\label{prop:P1nbratpoint}
Assume that $\mathcal{C}$ is defined over $k(t)=k(\PP^1)$, and that the affine equation $y^2=f(x)$ of $\mathcal{C}$ has coefficients in $k[t]$. Then for an integer $c>0$, there are at most
\begin{align*}
2^{\Omega\left(\left\lceil\frac{c}{2}\right\rceil,f,0\right) +\left\lceil\frac{c}{2}\right\rceil+1}
\end{align*}
points in the set $\mathcal{C}(k(t))_{\leq c}$ mapping (under $\delta$) to the same class in $H^1(C_f\setminus \pi^{-1}(\Sigma_2),\mu_2)$, where $\Sigma_2$ and $\delta$ are defined in Lemma~\ref{lem:descentmap}.
\end{proposition}

\begin{proof}
Let $(x_0,y_0)\in \mathcal{C}(k(t))_{\leq c}$. Since $k[t]$ is a unique factorization domain and $f$ is monic, with coefficients in $k[t]$, one can deduce from the relation $y_0^2=f(x_0)$ that
$$
x_0=\frac{u_0}{e^2}
$$
where $u_0$ and $e$ are coprime polynomials, unique up to multiplication by a scalar. More precisely, if $v$ is a valuation of $k[t]$ such that $v(x_0)<0$, then $v(y_0^2)=v(f(x_0))=v(x_0^d)=dv(x_0)$ (since $x_0^d$ is the leading term in $f(x_0)$), hence $v(x_0)$ is even (since $d$ is odd).

It follows that $\max\{\deg(u_0),2\deg(e)\} = \deg(x_0) \leq c$. On the curve $C_f$,
\begin{align*}
\divisor(e^2x_0-e^2x) &\geq -c\pi^*(\infty) -(x)_\infty \\
                      &\geq -2\left\lceil\frac{c}{2}\right\rceil\pi^*(\infty) -(x)_\infty,
\end{align*}
improving the inequality \eqref{eq:v0squared} (choosing $P_0=\infty$ and $D_0=\pi^*(\infty)$). This improvement allows us to replace $c+g$ by $\lceil\frac{c}{2}\rceil$ everywhere in the proof of Theorem~\ref{thm:nbratpoint}.
\end{proof}

Working again over the projective line, a classical problem is to count integral points, that is, points with coordinates in $k[t]$. The following statement gives a slightly better bound for the number of such points, provided $c$ is large enough. The main improvement is that we replace the variety $W'$ in the proof of Theorem~\ref{thm:nbratpoint} by a suitable linear subvariety of the target space.

\begin{proposition}
\label{prop:P1nbintpoint}
Under the assumptions of Proposition~\ref{prop:P1nbratpoint}, if $c\geq h_B(f)$ then the number of points in the set $\mathcal{C}(k[t])_{\leq c}$ mapping to the same class under $\delta$ is bounded above by
$$
\begin{cases}
2^{\max\left\{d\left\lceil\frac{c}{2}\right\rceil -g_f, \frac{d}{2}\left\lceil\frac{c}{2}\right\rceil\right\}+1} & \text{if $C_f$ is irreducible} \\
2^{d\left\lceil\frac{c}{2}\right\rceil + \omega(f)}  & \text{otherwise.}
\end{cases}
$$
\end{proposition}

\begin{proof}
Let us go through the proof of Theorem~\ref{thm:nbratpoint}, with the same notation. According to Proposition~\ref{prop:P1nbratpoint}, one can replace $c+g$ by $\lceil\frac{c}{2}\rceil$ when counting rational points.

Since $f$ is monic, with coefficients in $k[t]$, the function $x$ has all its poles in the support of $\pi^*(\infty)$. Therefore, if $x_0 \in k[t]_{\leq c}$ is a polynomial of degree $\leq c$, and if $c\geq h_B(f)=\deg(x)$, then we have, on the curve $C_f$,
$$
\divisor(x_0-x) \geq -c\pi^*(\infty).
$$
It follows that, when counting integral points, one can remove the quantity $h_B(f)$ from the upper bound on the integer $n$.

Next, we observe that, in order to count functions of the form $x_1-x$ with $x_1\in k[t]_{\leq c}$, the variety $W'$ can be replaced by the linear variety $\mathbb{P}(\langle x, k[t]_{\leq c}\rangle)$, and its subvariety $W_0'$ can be replaced by the affine subvariety $\mathbb{P}(\langle x, k[t]_{\leq c}\rangle)_0$ corresponding to functions of the form $\lambda x +\sum_{i\leq c} \mu_it^i$ with $\lambda\neq 0$.
Up to rescaling $\lambda$, such a function can be uniquely represented by a function of the form $x_1-x$ with $x_1\in k[t]_{\leq c}$. Therefore, the $k$-points of $W\cap \mathbb{P}(\langle x, k[t]_{\leq c}\rangle)_0$ are in bijection with functions of the form $x_1-x$ with $x_1\in k[t]_{\leq c}$ which define the same class as $x_0-x$ in $k(C_f)^\times/(k(C_f)^\times)^2$. It then follows from Lemma~\ref{x1lem} that $W\cap \mathbb{P}(\langle x, k[t]_{\leq c}\rangle)_0$ is a finite variety; according to B{\'e}zout's theorem, its degree is bounded above by $\deg W\leq 2^n$. It follows that the number of points in $\mathcal{C}(k[t])_{\leq c}$  having the same image by $\delta$ is bounded above by $2^{n+1}$.
\end{proof}


\subsection{Proof of Theorem~\ref{thm:intro1}}
\label{sec:intro1proof}

\begin{proof}[Proof of Theorem~\ref{thm:intro1} (1)]
Since $d$ is odd, $\mathcal{C}$ has a unique rational point at infinity, which induces an embedding $\mathcal{C}\hookrightarrow J$ of $\mathcal{C}$ into its Jacobian. It is well-known \cite{schaefer95} that the map $\delta$ is the composition of this particular embedding with the (cohomological) $2$-descent map on $J$. Therefore, the image of $\delta$ is a subset of the image of the Mordell-Weil group of $J$.

On the other hand, the canonical map $\Tr_{k(B)/k}(J)\to J$ is injective on $k(B)$-points \cite[Theorem~6.12]{conrad06}, so we have by construction of $\LN(J)$ an exact sequence of abelian groups
$$
0 \longrightarrow \Tr_{k(B)/k}(J)(k) \longrightarrow  J(k(B)) \longrightarrow  \LN(J) \to 0.
$$
Since $k$ is algebraically closed, the group $\Tr_{k(B)/k}(J)(k)$ is $2$-divisible, hence it follows from the Snake Lemma that
\begin{equation}
\label{eq:2torsLN}
J(k(B))[2]/\Tr_{k(B)/k}(J)(k)[2] \simeq \LN(J)[2],
\end{equation}
and that
$$
J(k(B))/2J(k(B)) \simeq \LN(J)/2\LN(J).
$$
It follows from the last statement that the size of the image of $\delta$ is bounded above by
$$
2^{\rk_\Z \LN(J) + \dim_{\F_2} \LN(J)[2]}.
$$
Combining this with Theorem~\ref{thm:nbratpoint} yields the bound.
\end{proof}

\begin{proof}[Proof of Theorem~\ref{thm:intro1} (2)]
Same proof as above, just replace Theorem~\ref{thm:nbratpoint} by its improved version over the projective line: Proposition~\ref{prop:P1nbratpoint}.
\end{proof}

\begin{proof}[Proof of \eqref{eq:GOSpimp}]
The dimension of $J$ is the genus of the curve $\mathcal{C}$ which is equal to $(d-1)/2$, so letting $d_0 :=\dim \Tr_{k(B)/k}(J)$ we have \cite[Th\'eor\`eme 3]{raynaud95}
\begin{equation}
\label{eq:GOSraw}
\rk_\Z \LN(J) \leq 4d_0 + (d-1)(2g-2)+ \deg(\mathfrak{f}_J),
\end{equation}
where $\mathfrak{f}_J$ denotes the conductor of $J$.

Since $k$ is algebraically closed of characteristic not $2$, the $2$-torsion subgroup of $\Tr_{k(B)/k}(J)(k)$ is an $\F_2$-vector space of dimension $2d_0$, hence it follows from \eqref{eq:2torsLN} and \eqref{eq:GOSraw} that
$$
\rk_\Z \LN(J) + \dim_{\F_2} \LN(J)[2] \leq 2d_0 + (d-1)(2g-2)+ \deg(\mathfrak{f}_J) + \dim_{\F_2} J(k(B))[2].
$$

Finally, we replace the size of the $2$-torsion subgroup by its value
$$
\dim_{\F_2} J(k(B))[2] = \omega(f)-1,
$$
where $\omega(f)$ is the number of irreducible factors of $f$. This concludes the proof.
\end{proof}


\section{Counting integral points}
\label{section3}


\subsection{Upper bound on the height of $S$-integral points}
\label{sec:intpointsheightbound}

In this section, we fix a finite non-empty set $S\subset B$, that we also view as a reduced divisor on $B$. We denote by $R_S\subset k(B)$ the ring of rational functions on $B$ with no poles outside $S$; we call it the ring of $S$-integers in $k(B)$.

In order to count the number of $S$-integral points on the Weierstrass model of $\mathcal{C}$, it suffices to give an upper bound on the height of such points, and then apply Theorem~\ref{thm:intro1}.

Inspired by the proof given by Hindry and Silverman \cite[Proposition~8.2]{HS} and its version in positive characteristic by Pacheco \cite{pacheco98}, we obtain the following.

\begin{theorem}
\label{thm:naiveheightbound}
Let $\mathcal{C}$ be the hyperelliptic curve over $k(B)$ defined by the equation $y^2=f(x)$, where $f \in k(B)[x]$ is a monic separable polynomial of odd degree $d\geq 3$.
Assume that $f$ has coefficients in $R_S$; let $\Delta_f$ be the discriminant of $f$, and let $\Sigma :=\{v\in B \mid v(\Delta_f) > 0\}$.
\begin{enumerate}
\item Assume $\car(k)=0$ and $\mathcal{C}$ is non-constant. Then we have, for all $(x_0,y_0)\in \mathcal{C}(R_S)$,
$$
\deg{x_0} \leq 4(2g-2+|S\cup \Sigma|) + \frac{3h_B(f)}{d}
$$
\item Assume $\car(k)>d$ and $\mathcal{C}$ is not isotrivial. Then
$$
\deg{x_0} \leq 6\rho(2g-2+|S\cup \Sigma|) + \frac{3h_B(f)}{d},
$$
where the inseparable degree $\rho$ is defined below (Definition~\ref{def:insepdeg}).
\end{enumerate}
\end{theorem}

Before we define $\rho$, let us recall basic facts about inseparable degrees. If $p=\car(k)>0$, then given $z\in k(B)\setminus k$ its inseparable degree is the largest power $p^s$ of $p$ such that $z$ belongs to $k(B)^{p^s}$; we denote it by $\degi(z)$. The separable degree is then defined by the formula
$$
\deg(z) = \degs(z)\degi(z).
$$
Alternatively, the separable (resp.\ inseparable) degree of $z$ is the separable (resp.\ inseparable) degree of the field extension $k(B)/k(z)$. When $z$ is a constant function, we do not define $\degs(z)$ and $\degi(z)$. We note that $\degi(z)$ does not change if one computes it over a separable extension $K/k(B)$, since
$$
\degi(K/k(z))= \degi(K/k(B))\cdot \degi(k(B)/k(z)) = \degi(k(B)/k(z)).
$$
The main property of $\degi$ we shall use in the proof is the following elementary one: if $z_1z_2$ is non-constant, then
\begin{equation}
\label{eq:insepdegprod}
\degi(z_1z_2) \geq \min \{\degi(z_1), \degi(z_2)\}.
\end{equation}

\begin{definition}
\label{def:insepdeg}
Assume $\car(k)>d$ and $\mathcal{C}$ is not isotrivial.
Given a finite separable extension $K/k(B)$ over which $f$ has all its roots $e_1,\dots, e_d$, we let
$$
\rho :=\inf\left\{\degi_K\left(\frac{e_k - e_1}{e_2 - e_1}\right) ~|~ k=3,\dots,d ~\text{and}~ \frac{e_k - e_1}{e_2 - e_1} \notin k \right\}.
$$
\end{definition}

Note that $\rho$ is well-defined, and strictly positive: if the set above were empty, then over $K$ the curve $y^2=(e_2-e_1)f(x)$ would be defined by an equation with coefficients in $k$, \emph{i.e.} $\mathcal{C}$ would be constant over a quadratic extension of $K$, hence isotrivial over $k(B)$. Moreover, as long as $K/k(B)$ is separable, this quantity does not depend on the choice of $K$.
Alternatively, one can define $\rho$ by
$$
\rho := \degi\left(K/k\left(\frac{e_k - e_1}{e_2 - e_1}; k=3,\dots,d\right)\right)
$$
from which one can check that $\rho$ does not depend on the choice of an ordering of $e_1,\dots,e_d$.

\begin{remark}
In the case when $\mathcal{C}$ is an elliptic curve, we have
$$
\rho = \degi\left(\frac{e_3 - e_1}{e_2 - e_1}\right) = \degi j(\mathcal{C})
$$
where $j(\mathcal{C})$ is the modular invariant. More generally, $\rho$ is the largest power of $p$ such that $\mathcal{C}$ is defined over $k(B)^\rho$.
\end{remark}

\begin{proof}
Let us assume first that $\car(k)=0$.
Let $P=(x_0,y_0)$ be an $S$-integral point. We shall work with the set $T:=S\cup \Sigma$, which is the smallest set containing $S$ and such that $\Delta_f$ is a $T$-unit. Let $e_1,\dots, e_d$ be the roots of $f$ in some algebraic closure of $k(B)$, ordered by increasing degree, let $u_i:=\sqrt{x_0-e_i}$, and let
$$
L:= k(B)(e_i,u_i,i=1,\dots,d).
$$
Since $\Delta_f$ is a $T$-unit, the extension $k(B)(e_1,\dots,e_d)/k(B)$, which is the splitting field of $f$, is unramified outside $T$. Moreover, the extension $L/k(B)(e_1,\dots,e_d)$ is unramified outside $T$, exactly by the same argument which proves that the descent map is well-defined (see Lemma~\ref{lem:descentmap}). Therefore, the extension $L/k(B)$ is unramified outside $T$.

Let $B'\to B$ be the finite cover of curves corresponding to the extension $L/k(B)$, and let $T'$ be the set of places of $B'$ lying over $T$. Since $B'\to B$ is unramified outside $T$, and tamely ramified above $T$, the Riemann-Hurwitz formula yields
\begin{equation}
\label{eq:RiemannHurwitz}
2g'-2+|T'| = [L:k(B)](2g-2+|T|)
\end{equation}
where $g'$ denotes the genus of $B'$.
It follows that, in the formula we want to prove, all quantities are multiplied by $[L:k(B)]$ when computed over $L$. So we may, and do, assume that $L=k(B)$.

We note that the $e_i$ are $T$-integers, because they are roots of a monic polynomial whose coefficients are $T$-integers. Likewise, since $x_0$ and the $e_i$ are $T$-integers, so are the $u_i$. Finally, since $\Delta_f$ is a $T$-unit, the $e_j-e_i$ are also $T$-units, for all $i\neq j$.

For appropriate choices of signs, we have the following relation between $T$-units:
\begin{equation}
\label{eq:relationu_i}
(u_1 \pm u_3) \pm (u_2 \pm u_3) = u_1 \pm u_2.
\end{equation}
Hence the \emph{abc}-theorem over function fields \cite{silverman84} implies that, for all choices of signs,
$$
\deg\left(\frac{u_1 \pm u_3}{u_1 \pm u_2}\right) \leq 2g-2+|T|.
$$
It follows that
$$
\deg\left(\frac{u_1}{u_1 \pm u_2}\right)=\deg\left(\frac{2u_1}{u_1 \pm u_2}\right) \leq \deg\left(\frac{u_1 + u_3}{u_1 \pm u_2}\right) + \deg\left(\frac{u_1 - u_3}{u_1 \pm u_2}\right) \leq 2(2g-2+|T|)
$$
(in passing, we used the fact that $2\neq 0$ in $k$). Therefore,
\begin{align*}
\deg\left(\frac{x_0-e_1}{e_2-e_1}\right) = \deg\left(\frac{u_1^2}{u_1^2-u_2^2}\right) &\leq \deg\left(\frac{u_1}{u_1 + u_2}\right) + \deg\left(\frac{u_1}{u_1 - u_2}\right) \\
                                &\leq 4(2g-2+|T|).
\end{align*}
On the other hand, by the properties of the degree we have
\begin{align*}
\deg(x_0)&=\deg\left(\frac{x_0-e_1}{e_2-e_1}(e_2-e_1)+e_1\right) \leq \deg\left(\frac{x_0-e_1}{e_2-e_1}\right) + 2\deg(e_1)+ \deg(e_2) \\
&\leq \deg\left(\frac{x_0-e_1}{e_2-e_1}\right) + \frac{3h_B(f)}{d},
\end{align*}
where the last inequality follows from the elementary observations (recalling the chosen ordering on the $e_i$):
$$
\deg(e_1)\leq \frac{h_B(f)}{d} \qquad \deg(e_1)+ \deg(e_2) \leq \frac{2h_B(f)}{d},
$$
hence the result.

Let us now consider the case when $\car(k)>d$. The first step of the proof (base-changing to $L$) is the same, observing that the extension $L/k(B)$ is separable (because $f$ is), and tamely ramified (since $\car(k)>d$), hence the Riemann-Hurwitz formula \eqref{eq:RiemannHurwitz} holds without change. Note that as previously all quantities in the formula are multiplied by the degree $[L:k(B)]$ and that $\rho$ remains unchanged.

We let as previously $e_1$ and $e_2$ be the roots of $f$ with smallest degree, then we choose a third root $e_3$ in such a way that
$$
\rho = \degi\left(\frac{e_3 - e_1}{e_2 - e_1}\right),
$$
where $\rho$ is the inseparable degree (Definition~\ref{def:insepdeg}).

Then the relation between $T$-units
$$
\frac{e_3 - e_1}{e_2 - e_1} + \frac{e_2 - e_3}{e_2 - e_1} =1
$$
implies, via the \emph{abc}-theorem, that
$$
\deg\left(\frac{e_3 - e_1}{e_2 - e_1}\right) = \degi\left(\frac{e_3 - e_1}{e_2 - e_1}\right)\degs\left(\frac{e_3 - e_1}{e_2 - e_1}\right) \leq \rho(2g-2+|T|).
$$

On the other hand, we have
\begin{equation}
\label{eq:unitvariation}
\frac{e_3 - e_1}{e_2 - e_1} = \left(\frac{u_1 - u_3}{u_1 - u_2}\right)\left(\frac{u_1 + u_3}{u_1 + u_2}\right).
\end{equation}
This quantity being non-constant, it follows from \eqref{eq:insepdegprod} that $\rho$ satisfies
$$
\min\left\{\degi\left(\frac{u_1 - u_3}{u_1 - u_2}\right),\degi\left(\frac{u_1 + u_3}{u_1 + u_2}\right)\right\} \leq \rho.
$$
Now, applying the \emph{abc}-theorem to the relation \eqref{eq:relationu_i} we have
$$
\degs\left(\frac{u_1 \pm u_3}{u_1 \pm u_2}\right) \leq 2g-2+|T|.
$$
Combining the two previous inequalities yields
$$
\min\left\{\deg\left(\frac{u_1 - u_3}{u_1 - u_2}\right),\deg\left(\frac{u_1 + u_3}{u_1 + u_2}\right)\right\} \leq \rho(2g-2+|T|).
$$
Assume that the first quantity is the minimum. Then, according to \eqref{eq:unitvariation} we have
$$
\deg\left(\frac{u_1 + u_3}{u_1 + u_2}\right) \leq \deg\left(\frac{e_3 - e_1}{e_2 - e_1}\right) + \deg\left(\frac{u_1 - u_3}{u_1 - u_2}\right) \leq 2\rho(2g-2+|T|).
$$

Finally, the same reasoning holds when switching signs between numerators in the right-hand side of \eqref{eq:unitvariation}. By the same method as in the characteristic $0$ case, we deduce that
$$
\deg\left(\frac{x_0-e_1}{e_2-e_1}\right) \leq 6\rho(2g-2+|T|),
$$
and the result follows by the same argument as in the characteristic zero case.
\end{proof}


\subsection{Proof of Theorem~\ref{thm:intro2}}
\label{sec:intro2proof}

\begin{proof}[Proof of Theorem~\ref{thm:intro2}]
With the notation of Theorem~\ref{thm:naiveheightbound} we have $|S\cup \Sigma|\leq |S|+\deg \Delta_f$, hence it follows from Theorem~\ref{thm:naiveheightbound} that the quantity $c_{\max}$ is an upper bound on the naive height of $S$-integral points on $\mathcal{C}$. It suffices to apply Theorem~\ref{thm:intro1} in order to deduce an upper bound on the number of $S$-integral points.

In order to prove Theorem~\ref{thm:intro2}, it remains to check that, in both brackets, the maximum is achieved by the first quantity, namely
\begin{equation}
\label{eq:firstmax}
\max\left\{d(c_{\max}+g)+h_B(f)-g_f, \frac{1}{2}\left(d(c_{\max}+g)+h_B(f)\right)\right\}
\end{equation}
when $C_f$ is irreducible, and
\begin{equation}
\label{eq:secondmax}
\max\left\{c_{\max}, \frac{1}{2}(c_{\max}+g)\right\}.
\end{equation}

By definition, $c_{\max}$ satisfies
\begin{equation}
\label{eq:cmaxLB}
c_{\max} \geq 4(2g - 2 + |T|) + \frac{3h_B(f)}{d},
\end{equation}
where $T := S\cup\Sigma$ as previously. If $g=0$ then $|T| \geq 2$ (since a non-constant fibration of the projective line has at least $2$ bad fibers), and in any case $|S|\geq 1$, hence we deduce from \eqref{eq:cmaxLB} that $c_{\max} \geq 4g$ in all cases. This proves the first quantity in \eqref{eq:secondmax} is the maximum.

In order to prove that the first quantity in \eqref{eq:firstmax} is the maximum, it suffices to prove that
$$
d(c_{\max}+g)+h_B(f) \geq 2g_f.
$$
But we have, according to the Riemann-Hurwitz formula,
$$
2g_f-2+|T_f| = d(2g-2+|T|),
$$
where $T_f$ denotes the set of points of $C_f$ lying above $T$. So, in order to prove the result, it suffices to prove that
$$
c_{\max}+g+\frac{h_B(f)}{d} \geq 2g-2+|T| + \frac{2}{d}.
$$
But this follows from \eqref{eq:cmaxLB}, observing that $h_B(f)\geq 1$ (if $h_B(f)=0$ then $f$ would have constant coefficients, which contradicts the assumption that $\mathcal{C}$ is non-constant).
\end{proof}


\subsection{Integral points of small height on elliptic curves over $\PP^1$}

In this section we work over $k(B)=k(\mathbb{P}^1)=k(t)$, although the arguments may admit extensions to the general case. Let $\mathcal{C}=\mathcal{E}$ be an elliptic curve defined over $k(t)=k(\PP^1)$ by an affine equation $y^2=f(x)$, where $\deg f=3$ and $f$ has coefficients in $k[t]$. In addition to the running assumption that $f$ is a monic separable cubic polynomial, we assume additionally that $f$ is irreducible (or equivalently, $C_f$ is irreducible). The key idea behind the results in this section is that when the divisor $E_0$ in the proof of Theorem~\ref{thm:nbratpoint} is a special divisor, then instead of applying Clifford's theorem we may use a more refined analysis based on classical results of Maroni on the Brill-Noether theory of trigonal curves. The resulting improvements will be important in the applications in the next section.

We begin by collecting some classical facts about trigonal curves. Let $C$ be a trigonal curve of genus $g>4$. This implies that the $g_3^1$ is unique, and we let $\mathfrak{g}_3^1$ be its image in $\Pic^3(C)$. We next recall the Maroni invariant $m$ of $C$, which we can take to be defined by \cite[Eq. (1.2)]{MS}
\begin{align*}
m=\min\{n\in\mathbb{N}\mid h^0(n\mathfrak{g}_3^1)>n+1\}-2.
\end{align*}
It is known \cite[Eq. (1.1)]{MS} that 
\begin{align*}
0<\frac{g-4}{3}\leq m\leq \frac{g-2}{2}.
\end{align*}
Let $W^r_n=W^r_n(C)=\{[D]\in \Pic^n(C)\mid \deg D=n, h^0(D)\geq r+1\}$, and let $W_n=W^0_n$. Let $\kappa=[K]\in \Pic^{2g-2}(C)$ be the canonical class of $C$.
We define
\begin{align*}
U_n^r=\begin{cases}
r\mathfrak{g}_3^1+W_{n-3r} &\text{if }n\geq 3r\\
\emptyset &\text{otherwise},
\end{cases}
\end{align*}
and
\begin{align*}
V_n^r&=\begin{cases}
\kappa-((g-n+r-1)\mathfrak{g}_3^1+W_{2(n-1)-g-3(r-1)}) &\text{if }2(n-1)-g-3(r-1)\geq 0\\
\emptyset &\text{otherwise}.
\end{cases}
\end{align*}

Then a classical result of Maroni \cite[Prop. 1]{MS} states:

\begin{theorem}
\label{BNtrigonal}
Let $C$ be a trigonal curve of genus $g>4$. For $n<g$ and $r\geq 1$, we have
\begin{enumerate}
\item  $W_n^r=U_n^r\cup V_n^r$\label{BNtrigonal1}
\item If $U_n^r\neq \emptyset$ then $U_n^r$ is an irreducible component of $W_n^r$.
\item Let $V_n^r\neq \emptyset$. Then $U_n^r\neq\emptyset$ and $V_n^r$ is an irreducible component of $W_n^r$ different from $U_n^r$ if and only if $g-n+r-1\leq m$.\label{BNtrigonal3}
\end{enumerate}
\end{theorem}

We reformulate this result in a form more convenient for our purposes.

\begin{corollary}
\label{h0bound2}
Let $C$ be a trigonal curve of genus $g>4$ and Maroni invariant $m$, and let $D$ be an effective divisor on $C$ with $\deg D<g$. Let $t\in k(C)$ be a rational function yielding the trigonal morphism. Then either $L(D)\subset k(t)$ or $h^0(D)\leq \max\left\{\min\left\{m+\deg D+2-g,\frac{2(\deg D-1)-g}{3}+2\right\},1\right\}$.
\end{corollary}

\begin{proof}
If $h^0(D)=1$ then the conclusion of the corollary is trivially satisfied. Suppose now that $h^0(D)=r+1\geq 2$ and let $n=\deg D$. We easily reduce to the case $|D|$ is basepoint free. By Theorem~\ref{BNtrigonal}\eqref{BNtrigonal1}, either $[D]\in U_n^r$ or $[D]\in V_n^r$. In the first case, we will show that $L(D)\subset k(t)$. In the second case, we combine two inequalities: one coming from the condition $V_n^r\neq\emptyset$, and one coming from requiring (as we may) $V_n^r\not\subset U_n^r$.

Suppose first that $[D]\in U_n^r$. Note that consistent with the fact $U_n^r\subset W_n^r$, since $n<g$ and $U_n^r\neq \emptyset$, we have $r\leq \frac{n}{3}\leq \frac{g-1}{3}\leq m+1$, and from our definition of the Maroni invariant, $rg_3^1$ is a complete linear system (and $\dim rg_3^1=r$). Then in this case, since $|D|$ is basepoint free, $[D]=r\mathfrak{g}_3^1$ and $|D|=rg_3^1$. Let $\phi\in L(D)$. Then it follows that we can write $\divisor(\phi)=E-F$, where $E,F\in rg_3^1$. But any such function $\phi$ can be written as a rational function in $t$ (of degree $\leq r$), and in particular, $L(D)\subset k(t)$.

Suppose now that $[D]\in V_n^r$, but $[D]\not\in U_n^r$. Then $V_n^r$ must be different from $U_n^r$, which implies that $g-n+r-1\leq m$ by Theorem~\ref{BNtrigonal}\eqref{BNtrigonal3}. Since $V_n^r\neq\emptyset$, we have $2(n-1)-g-3(r-1)\geq 0$. Combining these two inequalities gives
\begin{align*}
r+1=h^0(D)\leq \min\left\{m+n+2-g,\frac{2(n-1)-g}{3}+2\right\},
\end{align*}
completing the proof.
\end{proof}

Using Riemann-Roch and Serre duality when $\deg D\geq g$, Theorem~\ref{BNtrigonal} gives (see \cite[Remark~4.5(b)]{LN}:

\begin{theorem}
\label{h0trigonalbound}
Let $C$ be a trigonal curve of genus $g>4$ and let $D$ be a divisor. Then
\begin{align*}
h^0(D)\leq 
\begin{cases}
\frac{2}{3}(\deg D+2)-\frac{1}{3}g & \text{ if } g\leq \deg D\leq 2g-2,\\
\deg D+1-g & \text{ if } \deg D>2g-2.
\end{cases}
\end{align*} 
\end{theorem}

We now use these results to study integral points of small height on elliptic curves over $k(t)$.

\begin{theorem}
\label{generalMaronibound}
Let $\mathcal{E}$ be a curve over $k(t)=k(\PP^1)$ defined by an affine equation $y^2=f(x)=x^3+A(t)x+B(t)$, $A,B\in k[t]$. Let $C$ be the curve $C:x^3+A(t)x+B(t)=0$ and suppose that $C$ is irreducible and $g_C>4$. Let $m$ be the Maroni invariant of $C$. Let $c$ be an even integer satisfying
\begin{align*}
c\geq \frac{1}{2}\max\{\deg A,\deg B\}=\frac{1}{2}h(f).
\end{align*}
Let
\begin{align*}
\mu=
\begin{cases}
\min\left\{m+\frac{3c}{2}+2-g_C,c+\frac{4-g_C}{3}\right\} &\text{ if }c<\frac{2}{3}g_C\\
c+\frac{4-g_C}{3} &\text{ if }\frac{2}{3}g_C\leq c\leq \frac{4}{3}(g_C-1)\\
\frac{3}{2}c+1-g_C &\text{ if }c>\frac{4}{3}(g_C-1).
\end{cases}
\end{align*}
Then there are at most 
\begin{align*}
2^{\max\{\mu,1\}}
\end{align*}
points in the set $\mathcal{E}(k[t])_{\leq c}$ mapping (under $\delta$) to the same class in $H^1(C_f\setminus \pi^{-1}(\Sigma_2),\mu_2)$.
\end{theorem}

\begin{proof}
Let $(x_0,y_0)\in \mathcal{E}(k[t])_{\leq c}$. From the form of $f(x)$ and since $c\geq \frac{1}{2}h(f)$, we have
\begin{align*}
\divisor(x_0-x)\geq -c\pi^*\infty,
\end{align*}
and therefore we can remove the term $D_\infty$ in the proof of Theorem~\ref{thm:nbratpoint}. With the same notation as in the proof of Theorem~\ref{thm:nbratpoint}, we have
\begin{equation}
\label{eq:psi0bis}
\divisor(\psi_0) = \divisor(x_0-x) = 2E_0+\sum_{i=1}^s P_i -c\pi^*(\infty)
\end{equation}
(note also that by our assumptions $2\mid c$, so that $2\lceil\frac{c}{2}\rceil=c$).
It follows from \eqref{eq:psi0bis} that $2\deg(E_0) \leq c\deg(\pi^*(\infty))$, or equivalently,
\begin{align*}
\deg E_0\leq \frac{3c}{2}.
\end{align*}
Suppose first that $L(E_0)\not\subset k(t)$. Then using Corollary~\ref{h0bound2} or Theorem~\ref{h0trigonalbound} (depending on $\deg E_0$), we find  $h^0(E_0)\leq \mu$. Then the same proof as in Proposition~\ref{prop:P1nbintpoint} proves that $\delta$ is at most $2^\mu$-to-one on such points.

Let us now consider the case when $L(E_0)\subset k(t)$. By \eqref{eq:psi0bis}, the map
\begin{align*}
L(c\pi^*(\infty)-\sum_{i=1}^sP_i-E_0)&\to L(E_0)\\
\phi'&\mapsto \phi'/\psi_0, 
\end{align*}
is an isomorphism. Applying it to the map $\phi$ defined by $\psi_0\psi_1=\phi^2$ as in the proof of Theorem~\ref{thm:nbratpoint}, we deduce that $\phi/\psi_0\in k(t)$ and so $\psi_1=(\phi/\psi_0)^2\psi_0$ differs from $\psi_0$ by the square of a rational function in $t$. But since $1$ and $x$ are linearly independent over $k(t)\subset k(C)$, looking at coefficients this immediately implies that $\psi_0=\psi_1$, hence $\delta$ is $2$-to-one on such points.
\end{proof}

We get a more precise result in the $j$-invariant $0$ case.

\begin{theorem}
\label{j0bound}
Let $\mathcal{E}$ be a curve over $k(t)=k(\PP^1)$ defined by an affine equation $y^2=f(x)=x^3+B(t)$.  Write
\begin{align*}
B=B_1B_2^2B_3^3,
\end{align*}
where $B_1$ and $B_2$ are coprime and squarefree. Let $d_1=\deg B_1$ and $d_2=\deg B_2$.  Let 
\begin{align*}
g_C=
\begin{cases}
d_1+d_2-2 &\text{if $3| (d_1+2d_2)$},\\
d_1+d_2-1 &\text{otherwise},
\end{cases}
\end{align*}
and assume $g_C>4$.
Let $c$ be an even integer satisfying
\begin{align*}
\frac{1}{3}\deg B\leq c<\frac{2}{3}g_C.
\end{align*}
Then there are at most 
\begin{align*}
2^{\max\{\lceil\frac{1}{3}\min\{d_1+2d_2,2d_1+d_2\}\rceil+\frac{3c}{2}-g_C,1\}}
\end{align*}
points in the set $\mathcal{E}(k[t])_{\leq c}$ mapping (under $\delta$) to the same class in $H^1(C_f\setminus \pi^{-1}(\Sigma_2),\mu_2)$.
\end{theorem}

\begin{proof}
Since $g_C>4$, $B$ is not a cube and the curve $C:x^3+B(t)=0$ is irreducible. We note that
\begin{align*}
3\divisor(x)=\divisor(B),
\end{align*}
as divisors on $C$, and the divisor of poles of $x$ is given by $\frac{\deg B}{3}\pi^*(\infty)$ (note that either $3|\deg B$ or $\pi^*\infty$ is a point with multiplicity $3$). Let $(x_0,y_0)\in \mathcal{E}(k[t])_{\leq c}$.  Since $c\geq \frac{1}{3}\deg B$, we again have
\begin{align*}
\divisor(x_0-x)\geq -c\pi^*\infty.
\end{align*}
The formula for the genus $g_C$ is well-known. Finally, we note that $x/B_3$ and $x^2/(B_2B_3^2)$ have poles only at infinity, and $\deg(x/B_3)=d_1+2d_2$, $\deg(x^2/(B_2B_3^2))=2d_1+d_2$. Since $x,x^2\not\in k(t)$ (viewing $k(t)\subset k(C)$), it follows that the Maroni invariant $m$ satisfies $m\leq \lceil\frac{1}{3}\min\{d_1+2d_2,2d_1+d_2\}\rceil-2$. The result now follows from the previous result.
\end{proof}

For elliptic curves of the form $\mathcal{E}:y^2=x^3+B(t)$, we may use Davenport's inequality to give an improvement of Theorem~\ref{thm:intro2}, and a generalization of our earlier work \cite{GHL23}.

\begin{corollary}
\label{thm:isotrivialcase}
With the notation of Theorem~\ref{j0bound}, assume that $g_C>4$, and that either $\car(k)=0$ or $\car(k) > 3$ and $B$ contains at least one root of multiplicity one. Then
$$
|\mathcal{E}(k[t])| \leq 2^{3\deg(B)-2-g_C+\rk_\Z \mathcal{E}(k(t))}.
$$
\end{corollary}

\begin{proof}
According to Davenport's inequality \cite{davenport65}, generalized to positive characteristic by Sch{\"u}tt and Schweizer \cite[Theorem~1.2, (b)]{SS}, for all points $(x_0,y_0) \in \mathcal{E}(k[t])$ we have
$$
\deg x_0 \leq 2\deg(B)-2.
$$
Using the explicit formula for $g_C$ one checks that $2\deg(B)-2 > \frac{4}{3}(g_C-1)$. Now, it follows from Theorem~\ref{generalMaronibound} that the $2$-descent map is at most a $2^{3\deg(B)-2-g_C}$-to-one map on the set of all integral points, hence the result.
\end{proof}


\section{Applications to bounding torsion of Jacobians over small finite fields}

\subsection{Bounding $3$-torsion of Jacobians of hyperelliptic curves over small finite fields}
\label{section4}


Let $q=p^r$ for some prime $p\geq 5$.
Let $X$ be a hyperelliptic curve of genus $g$ over $\mathbb{F}_q$, with a rational Weierstrass point. Then one can find an equation $X:y^2=F(t)$  with $F$ squarefree, $\deg F=d=2g+1$. Let $J=\Jac(X)$ be the Jacobian of $X$.
The main result of this section is Theorem~\ref{thm:intro3}: for some constant $\gamma$ depending only on $q$,
\begin{align*}
|J(\mathbb{F}_q)[3]|\leq q^{\frac{g}{2}+\gamma\frac{g}{\log g}}.
\end{align*}

The starting point of our strategy is to relate the $3$-torsion points of $J(\mathbb{F}_q)$ to integral points on certain elliptic curves over $\mathbb{F}_q(t)$.

Actually we prove a slightly more general statement regarding $n$-torsion points of the Jacobian of $X$ over an arbitrary base field $k$.

\begin{lemma}
\label{lemma:maintrick}
Let $n>2$ be an odd integer. For $a\in k[t], a\neq 0$, let $\mathcal{C}_a$ be the hyperelliptic curve (over $k(t)$) defined by
\begin{align*}
y^2=x^n+a(t)^2F(t).
\end{align*} 
Then there exists an injective map
\begin{align*}
J(k)[n]\setminus\{0\}\to \bigsqcup_{\substack{\deg a\leq \frac{(n-2)g-1}{2}\\a\neq 0}} \mathcal{C}_a(k[t])_{\leq g}.
\end{align*}
\end{lemma}

\begin{proof}
Let $\infty$ denote the (unique) point at infinity on $X$. Then we may write $P=[D-g\infty]$ for some effective divisor $D$ on $X$ of degree $g$, and 
\begin{align*}
nD-ng\infty=\divisor(\phi),
\end{align*}
for some rational function $\phi\in k(X)$. Since $\phi$ has poles only at infinity, we may write
\begin{align*}
\phi=a(t)y+b(t),
\end{align*}
for some polynomials $a,b\in k[t]$. Using that $n$ is odd, elementary arguments yield $a\neq 0$. Since $\deg y=2g+1, \deg t=2$, $\deg D=g$, we find that $\max\{2\deg_t a+2g+1,2\deg_t b\}\leq ng$ (noting that the two integers in the maximum have opposite parity). In particular, this implies $\deg_ta\leq \frac{(n-2)g-1}{2}$. Taking norms gives
\begin{align*}
b^2-a^2y^2=b^2-a^2F=c(x_0)^n
\end{align*}
for some $x_0\in k[t], c\in k^*$. From previous calculations, $n\deg_t x_0\leq ng$ and so $\deg_t x_0\leq g$. Then $(cx_0,c^{(n-1)/2}b)$ is a point on the hyperelliptic curve
\begin{align*}
\mathcal{C}_{c^{\frac{n-1}{2}}a(t)}:y^2=x^n+(c^{\frac{n-1}{2}}a(t))^2F(t).
\end{align*}
This construction gives a map as in the statement of the lemma (dependent on some arbitrarily made choices; for instance, $\phi$ is only determined up to a constant). To show the map is injective, we note that if $(x_0,y_0)\in \mathcal{C}_a(k[t])$ is in the image of the map, then $P$ is determined as $P=\frac{1}{n}\divisor(a(t)y+y_0(t))$.
\end{proof}

Let us point out that although Lemma~\ref{lemma:maintrick} holds for all $k$, when $\car(k)$ divides $n$, the equation $y^2=x^n+a(t)^2F(t)$ does not define a smooth curve over $k(t)$ (see Remark~\ref{remark:char3}).

We now return to a finite field $k=\mathbb{F}_q$, where $q$ is the power of a prime $p\geq 5$, and $n=3$.
We shall need the following Lemma, which is closely related to the prime number theorem over $\mathbb{F}_q[t]$.

\begin{lemma}
\label{irreduciblefactors}
Let $F(t)\in \mathbb{F}_q[t]$ be a squarefree polynomial of degree $d$. Then $F(t)$ has at most $4q\frac{d}{\log_q d}$ irreducible factors.
\end{lemma}

\begin{proof}
Let $p_1,p_2,\ldots,p_{N(m)}$ be the monic irreducible polynomials in $\mathbb{F}_q[t]$ of degree at most $m$.  We first note that there are at most $q^n/n$ monic irreducible polynomials in $\mathbb{F}_q[t]$ of degree $n$. Indeed, this is immediate from the observation that the $n$ roots of each such polynomial yield $n$ distinct elements of $\mathbb{F}_{q^n}$. In particular, $N(m)\leq \sum_{i=1}^m \frac{q^i}{i}$.

Let $\omega(F)$ be the number of irreducible factors of $F$. Then clearly if $\sum_{i=1}^{N(m)}\deg p_i\geq d$, then $\omega(F)\leq N(m)$. Since $x^{q^m}-x$ is the product of the distinct monic irreducible polynomials in $\mathbb{F}_q[t]$ of degree dividing $m$, we have in particular
\begin{align*}
\sum_{i=1}^{N(m)}\deg p_i\geq q^m.
\end{align*}
So if $m=\lceil\log_q(d)\rceil\leq \log_q(d)+1$, then
\begin{align*}
\omega(F)\leq N(m)\leq \sum_{i=1}^m \frac{q^i}{i}\leq 4\frac{q^m}{m}, 
\end{align*}
where the last inequality is easily proven by induction (assuming $q\geq 2$). Thus, 
\begin{align*}
\omega(F)\leq 4\frac{qd}{\log_q(d)+1}.
\end{align*}
\end{proof}

In order to prove Theorem~\ref{thm:intro3} we shall use the following result of Brumer \cite[Proposition~6.9]{brumer92} to bound the rank of an elliptic curve $E$ over $\mathbb{F}_q(t)$:
\begin{equation}
\label{eq:Brumer}
\rk_{\Z} E(\mathbb{F}_q(t))\leq \frac{\deg(\mathfrak{f}_E)-4}{2\log_q \deg(\mathfrak{f}_E)}+ \lambda\frac{\deg(\mathfrak{f}_E)}{(\log_q \deg(\mathfrak{f}_E))^2},
\end{equation}
where $\mathfrak{f}_E$ is the conductor of $E$ and $\lambda$ is a constant depending only on $q$, which has been made explicit by Pazuki \cite{pazuki2022}. The characteristic $p$ is assumed to be at least $5$ here.

\begin{proof}[Proof of Theorem~\ref{thm:intro3}]
According to Lemma~\ref{lemma:maintrick}, each nonzero element of $J(\mathbb{F}_q)[3]$ gives rise to a distinct point $(x_0,y_0)\in \mathcal{E}_a(k[t])$ on an elliptic curve 
\begin{align*}
\mathcal{E}_a:y^2=x^3+a(t)^2F(t), 
\end{align*} 
for some $a\in \mathbb{F}_q[t]$, $a\neq 0$, $\deg a\leq \frac{d-3}{4}$, and $\deg x_0\leq \frac{d-1}{2}$. Thus,
\begin{align*}
|J(\mathbb{F}_q)[3]|-1\leq \sum_{\substack{a\in \mathbb{F}_q[t]\setminus\{0\}\\ \deg a\leq \frac{d-3}{4}}}\left|\mathcal{E}_a(k[t])_{\leq \frac{d-1}{2}}\right|.
\end{align*} 
We now bound the right-hand side of this inequality.

Let $c=2\lceil\frac{d-1}{4}\rceil\in \{\frac{d-1}{2},\frac{d+1}{2}\}$ and let $a\in \mathbb{F}_q[t]\setminus\{0\}$ with $\deg a\leq \frac{d-3}{4}$. Then
\begin{align*}
\frac{1}{3}\deg (a^2F)\leq \frac{1}{3}\left(\frac{d-3}{2}+d\right)\leq c.
\end{align*}
Let $C_a$ be the curve over $\mathbb{F}_q$ given by $C_a:x^3+a(t)^2F(t)=0$. We may write $a(t)$ in the form
\begin{align*}
a(t)=a_0a_1a_2^2a_3^3,
\end{align*}
where $a_0=\gcd(a,F)$, $a_0,a_1,a_2$ are squarefree, and $a_1$ and $a_2$ are coprime.  Let $d_i=\deg a_i, i=0,1,2,3$. Then the genus of $C_a$ is given by
\begin{align*}
g_{C_a}=d-d_0+d_1+d_2-\epsilon_a,
\end{align*}
where $\epsilon_a\in \{1,2\}$ (with the value depending on if $3|(\deg a^2F)$).

\emph{Case 1:} We first assume that $d_0< \frac{d-3}{4}-2=\frac{d-11}{4}$ (note that in any case $d_0\leq \deg a\leq \frac{d-3}{4}$). Then one easily finds
\begin{align*}
c<\frac{2}{3}g_{C_a}.
\end{align*}
By Theorem~\ref{j0bound}, letting $r_a$ be the rank of $\mathcal{E}_a(\mathbb{F}_q(t))$, we have
\begin{align*}
\left|\mathcal{E}_a(\mathbb{F}_q[t])_{\leq c}\right|&\leq 2^{\frac{1}{3}(d-d_0+2d_1+d_2)+\frac{3}{2}\frac{d+1}{2}- (d-d_0+d_1+d_2-\epsilon_a)+r_a}\\
&\leq 2^{\frac{1}{12}d+\frac{2}{3}d_0-\frac{1}{3}d_1-\frac{2}{3}d_2+\epsilon_a+r_a+1}.
\end{align*}
Note that $\frac{1}{3}d_1+\frac{2}{3}d_2\leq \frac{1}{3}\deg a\leq \frac{d-3}{12}$, and so the quantity in the exponent is always at least $1$.

Let $r=\max_{\deg a\leq \frac{d-3}{4}}r_a$. If one fixes a polynomial $F_0$ of degree $d_0$, and integers $d_i$, $i=1,2,3$ with $d_0+d_1+2d_2+3d_3\leq \frac{d-3}{4}$, then the number of polynomials $a$ of degree at most $\frac{d-3}{4}$ with a factorization $a_0=F_0=\gcd(a,F)$ and $a_i$ of degree $d_i$ is at most
\begin{align*}
q^{(d_1+1)+(d_2+1)+(d_3+1)}.
\end{align*}
Since $2<q$ and 
\begin{align*}
\frac{2}{3}d_0+\frac{2}{3}d_1+\frac{1}{3}d_2+d_3\leq \frac{2}{3}(d_0+d_1+2d_2+3d_3)\leq \frac{2}{3}\deg a\leq \frac{2}{3}\frac{d-3}{4}=\frac{d-3}{6},
\end{align*}
we have
\begin{align*}
\sum_{\substack{a\in\mathbb{F}_q[t]\setminus\{0\}\\ \deg a_i=d_i, i=1,2,3\\a_0=F_0}}    \left|\mathcal{E}_a(\mathbb{F}_q[t])_{\leq c}\right|&\leq 2^{\frac{1}{12}d+\frac{2}{3}d_0-\frac{1}{3}d_1-\frac{2}{3}d_2+\epsilon_a+r+1}q^{d_1+d_2+d_3+3}\\
&\leq q^{\frac{1}{12}d+\frac{2}{3}d_0+\frac{2}{3}d_1+\frac{1}{3}d_2+d_3+r+6}\\
&\leq q^{\frac{1}{12}d+\frac{d-3}{6}+r+6}\\
&\leq q^{\frac{d}{4}+r+6}.
\end{align*}

There are at most $d$ possibilities for each of the integers $d_1,d_2,d_3$, and by Lemma~\ref{irreduciblefactors}, there are at most $2^{\frac{4qd}{\log_qd}}$ possibilities for $F_0$. Therefore,
\begin{align*}
\sum_{\substack{a\in\mathbb{F}_q[t]\setminus\{0\}\\  \deg a\leq \frac{d-3}{4}\\ d_0\leq \frac{d-11}{4}}}    \left|\mathcal{E}_a(\mathbb{F}_q[t])_{\leq c}\right|&\leq d^32^{\frac{4qd}{\log_qd}}q^{\frac{d}{4}+r+6}.
\end{align*}
  
\emph{Case 2:} Similarly, there are at most $2^{\frac{4qd}{\log_qd}}q^3$ polynomials $a$ with $\deg a\leq \frac{d-3}{4}$ and $d_0\geq \frac{d-11}{4}$. Since $c\leq \frac{d+1}{2}$ and $g_{C_a}\geq d-d_0-2\geq \frac{3d}{4}$, by Theorem~\ref{generalMaronibound}, we have
\begin{align*}
\sum_{\substack{a\in\mathbb{F}_q[t]\setminus\{0\}\\  \deg a\leq \frac{d-3}{4}\\ d_0\geq \frac{d-11}{4}}}    \left|\mathcal{E}_a(\mathbb{F}_q[t])_{\leq c}\right|&\leq q^32^{c-\frac{g_{C_a}}{3}+2+\frac{4qd}{\log_qd}+r}\\
&\leq q^32^{\frac{d+1}{2}-\frac{d}{4}+2+\frac{4qd}{\log_qd}+r}\leq q^{\frac{d}{4}+6+\frac{4qd}{\log_qd}+r}.
\end{align*}

Finally, from the shape of the equation of $\mathcal{E}_a$ we see that $\deg(\mathfrak{f}_{E_a})\leq 3d$, hence by the result of Brumer \eqref{eq:Brumer} we have
$$
r\leq \frac{3d}{2\log_q(3d)}+ \lambda\frac{3d}{(\log_q (3d))^2}
$$
for some constant $\lambda$.
Combining everything, it follows that for some constant $\gamma$ we have
\begin{align*}
\sum_{\substack{a\in\mathbb{F}_q[t]\setminus\{0\}\\  \deg a\leq \frac{d-3}{4}}}    \left|\mathcal{E}_a(\mathbb{F}_q[t])_{\leq c}\right|
\leq q^{\frac{d}{4}+\gamma\frac{d}{\log d}}
\end{align*}
and the result follows since $d=2g+1$.
\end{proof}

\begin{remark}
\label{remark:Xbadred}
Let $p\geq 5$ be a prime number, and let $X$ be a hyperelliptic curve of genus $g$ defined over $\Q_p$, with a rational Weierstrass point. Let $\mathcal{J}\to\Spec(\Z_p)$ be the N\'eron model of $\Jac(X)$, and let $\mathcal{J}_p$ be the fiber of $\mathcal{J}$ at $p$.

Since $p\neq 3$, the $3$-torsion subgroup scheme $\mathcal{J}[3]$ is {\'e}tale (but not necessarily finite) over $\Spec(\Z_p)$, hence the reduction map
$$
\Jac(X)(\Q_p)[3]=\mathcal{J}(\Z_p) \to \mathcal{J}_p(\F_p)[3]
$$
is injective (and in fact bijective by Hensel's Lemma). So, bounding $\Jac(X)(\Q_p)[3]$ is equivalent to bounding $\mathcal{J}_p(\F_p)[3]$.

According to \cite{D2M2}, the special fiber $\mathcal{X}_p$ of the minimal regular model $\mathcal{X}\to \Spec(\Z_p)$ of $X$ can be explicitly described from the so-called ``cluster picture'' attached to the roots of $F$ in an equation $X:y^2=F(x)$. If $X$ has semi-stable reduction, the special fiber of $\mathcal{X}$ is reduced, and consists of hyperelliptic curves $X_1,\dots,X_r$ (possibly singular) linked by chains of $\PP^1s$. The connected component $\mathcal{J}_p^0$ of $\mathcal{J}_p$ is in this case an extension of an abelian variety $B$ by a torus $T$, whose rank we denote by $t$. The well-known relation \cite{raynaud70} between $\mathcal{J}$ and the relative Picard functor of $\mathcal{X}$ implies that the abelian variety $B$ is the product of Jacobians of the $\tilde{X_i}$ (the desingularized $X_i$), and in particular we have $g=g(\tilde{X_i})+\dots +g(\tilde{X_i})+t$. By virtue of Theorem~\ref{thm:intro3}, and since $T[3]$ is an {\'e}tale group scheme of rank $3^t$ over $\F_p$, we deduce that
$$
|\mathcal{J}_p^0(\F_p)[3]| \leq  3^t \times p^{\frac{g-t}{2}+\gamma\frac{g-t}{\log g_{\min}}}
$$
where $\gamma$ is an absolute constant, and $g_{\min}:=\min_i g(\tilde{X_i})$.

The remaining contribution comes from the group of components $\Phi:=\mathcal{J}_p/\mathcal{J}_p^0$. Since the reduction is semi-stable, it is known at least since Grothendieck \cite[expos\'e~IX, 11.9, 11.11]{SGA7-1} that the {\'e}tale group scheme $\Phi$ is generated by at most $t$ elements. Therefore, $\Phi[3](\_p)$ has cardinality at most $3^t$. Putting everything together we deduce that, in the semi-stable case,
$$
|\mathcal{J}_p(\F_p)[3]| \leq 3^{2t}\times p^{\frac{g-t}{2}+\gamma\frac{g-t}{\log g_{\min}}}.
$$
We leave other cases to the interested reader.
\end{remark}

Previous authors have given bounds on the torsion subgroup of an abelian variety by taking advantage of bad reduction, an approach which seems orthogonal to ours. For example, Clark and Xarles \cite{CX2008} give torsion bounds for an abelian variety with purely additive reduction over a $p$-adic field. In \cite{Lorenzini2011}, Lorenzini studies the ratio between the product of the Tamagawa numbers and the torsion subgroup of an abelian surface.

\begin{remark}
\label{remark:char3}
Assume that $\car(k)=3$. Then Lemma~\ref{lemma:maintrick} still holds: a $3$-torsion point gives rise to a $k[t]$-integral point on some curve $\mathcal{E}_a:y^2=x^3+a(t)^2F(t)$. Now, since the characteristic of the base field is $3$, this curve $\mathcal{E}_a$ has arithmetic genus $1$ but geometric genus $0$; more precisely, it becomes rational over the field $k(\sqrt[3]{a^2F})$, hence has infinitely many integral points over this field. In short, $\mathcal{E}_a$ is not an elliptic curve and the strategy of our proof is irrelevant in this case.
\end{remark}

\subsection{Bounding $2$-torsion of Jacobians of trigonal curves over small finite fields}
\label{section5}

In this section we interchange the roles of $2$ and $3$, and prove analogues of results of the previous section for $2$-torsion of Jacobians of trigonal curves.

Let $q=p^r$ for some prime $p\geq 5$.
Let $X$ be a trigonal curve of genus $g$ over $\mathbb{F}_q$ with trigonal morphism $\pi:X\to\mathbb{P}^1$. We assume that there exists a totally ramified rational point $P_\infty\in X(\mathbb{F}_q)$ of $\pi$, in which case after an automorphism of $\mathbb{P}^1$ we may assume that $\pi^*(\infty)=3P_\infty$. Let $J=\Jac(X)$ be the Jacobian of $X$.
The main result of this section is Theorem~\ref{thm:intro4}: for some constant $\gamma$ depending only on $q$,
\begin{align*}
|J(\mathbb{F}_q)[2]|\leq (2q)^{\frac{g}{3}+\gamma\frac{g}{\log g}}.
\end{align*}

As in the previous section, we first relate $2$-torsion points of $J(\mathbb{F}_q)$ to integral points on certain elliptic curves over $\mathbb{F}_q(t)$.
\begin{lemma}
\label{lemma:maintrick2}
There exists a set $\mathbf{E}$ of elliptic curves over $\mathbb{F}_q(t)$ (depending on the trigonal curve $X$) with the following properties:
\begin{enumerate}
\item $|\mathbf{E}|=q^{\lceil \frac{g}{3}\rceil}-1$
\item Each elliptic curve $\E\in \mathbf{E}$ may be defined by a Weierstrass equation 
\begin{align*}
y^2=x^3+a_2(t)x^2+a_1(t)x+a_0(t), 
\end{align*}
where $a_i\in \mathbb{F}_q[t]$, $\deg a_i\leq \frac{2g}{3}(3-i)$, and $x^3+a_2(t)x^2+a_1(t)x+a_0(t)$
 is irreducible in $\mathbb{F}_q[t,x]$ and defines a nonsingular projective curve isomorphic to $X$ over $\mathbb{F}_q$ (and in particular of genus $g$).
\item There is a map
\begin{align*}
J(\mathbb{F}_q)[2]\setminus\{0\}\to \bigsqcup_{\E\in \mathbf{E}} \E(\mathbb{F}_q[t])_{\leq \frac{2g}{3}},
\end{align*}
which is at most $3$-to-$1$.
\label{item:injectivemap}
\end{enumerate}
\end{lemma}

\begin{proof}

We first construct an appropriate set of elliptic curves $\mathbf{E}$ over $\mathbb{F}_q(t)$. By assumption, $\pi^*(\infty)=3P_\infty$, where $P_\infty\in X(\mathbb{F}_q)$.  We note that by Riemann-Roch, $h^0(2gP_\infty)=2g+1-g=g+1$. The map $\pi$ induces an inclusion of function fields $\mathbb{F}_q(\mathbb{P}^1)=\mathbb{F}_q(t)\subset \mathbb{F}_q(X)$, and viewed as a function on $X$ we have $\deg t=3$ and $1,t,\ldots, t^{\lfloor 2g/3\rfloor}\in L(2gP_\infty)$. Note that 
\begin{align*}
h^0(2gP_\infty)-(\lfloor 2g/3\rfloor+1)=g-\lfloor 2g/3\rfloor=\lceil g/3\rceil. 
\end{align*}
We complete $1,t,\ldots, t^{\lfloor 2g/3\rfloor}$ to a basis $1,t,\ldots, t^{\lfloor 2g/3\rfloor}, \psi_1,\ldots, \psi_{\lceil g/3\rceil}$ of $L(2gP_\infty)$ over $\mathbb{F}_q$. Let $V\subset \mathbb{F}_q(X)$ be the $\mathbb{F}_q$-vector space spanned by $\psi_1,\ldots, \psi_{\lceil g/3\rceil}$. Let $\psi\in V\setminus\{0\}$, and let $F_\psi\in \mathbb{F}_q(t)[x]$ denote the minimal polynomial of $\psi$ over $\mathbb{F}_q(t)$. By construction, $\psi\not\in \mathbb{F}_q(t)$, and since $\mathbb{F}_q(t)\subset \mathbb{F}_q(t,\psi)\subset \mathbb{F}_q(X)$ and $[\mathbb{F}_q(X):\mathbb{F}_q(t)]=3$ is prime, we must have $\mathbb{F}_q(t,\psi)=\mathbb{F}_q(X)$ and $\deg F_\psi=3$. Let
\begin{align*}
F_\psi=x^3+a_{2,\psi}x^2+a_{1,\psi}x+a_{0,\psi},
\end{align*}
where each coefficient $a_{i,\psi}$ is an appropriate symmetric polynomial in the conjugates of $\psi$ over $k(t)$. Since $\psi\in L(2gP_\infty)$ has poles only at $P_\infty$, it follows that $a_{i,\psi}\in \mathbb{F}_q[t]$, and looking at the order of poles we immediately find $\deg_t a_{i,\psi}\leq \frac{2g}{3}(3-i)$.  This can also be proved by considering the Newton polygon of $F_\psi$, which has a unique edge since $P_\infty$ is totally ramified. Note that $F_\psi$ is irreducible in $\mathbb{F}_q[t,x]$, and the (nonsingular projective) curve over $\mathbb{F}_q$ defined by $F_\psi=0$ is isomorphic to $X$. 

We let $\E_\psi$ be the elliptic curve over $k(t)$ given by the Weierstrass equation
\begin{align*}
y^2=F_\psi=x^3+a_{2,\psi}(t)x^2+a_{1,\psi}(t)x+a_{0,\psi}(t),
\end{align*}
and we let $\mathbf{E}=\{\E_\psi\mid \psi \in V\setminus \{0\}\}$. Since $\dim V=\lceil g/3\rceil$, we have $|\mathbf{E}|=q^{\lceil g/3\rceil}-1$. It remains only to prove \eqref{item:injectivemap}.

Let $P\in J(\mathbb{F}_q)[2]\setminus\{0\}$. We may write $P=[D-gP_\infty]$ for some effective divisor $D$ on $X$ of degree $g$, and then
\begin{align*}
2D-2gP_\infty=\divisor(\phi),
\end{align*}
for some rational function $\phi\in L(2gP_\infty)\subset k(X)$. Then from the definitions, we may write $\phi=x_0(t)-\psi$ for some  $x_0\in \mathbb{F}_q[t]$, $\deg x_0\leq \lfloor 2g/3\rfloor$, and $\psi\in V$. Moreover, $\psi=0$ easily implies that $x_0$ is a constant multiple of a square in $\mathbb{F}_q[t]$ and $P=[0]$. Therefore, we must have $\psi\in V\setminus\{0\}$. Taking norms, we find by the same argument as in the proof of Lemma~\ref{x1lem} that
\begin{align*}
\Norm_{X/\PP^1}(\phi)=\Norm_{X/\PP^1}(x_0(t)-\psi)=F_\psi(x_0(t)),
\end{align*}
and, on the other hand, $\Norm_{X/\PP^1}(\phi)= cy_0(t)^2$ 
for some $y_0\in \mathbb{F}_q[t]$ and $c\in \mathbb{F}_q^*$. Replacing $x_0$ by $cx_0$, $\psi$ by $c\psi$, and $y_0$ by $c^2y_0$, we may assume $c=1$. Then we obtain a point $(x_0,y_0)\in \E_\psi(k[t])_{\leq \frac{2g}{3}}$, where $\E_\psi\in \mathbf{E}$. Conversely, for a point $(x_0,y_0)\in \E_\psi(k[t])$ in the image of the map in \eqref{item:injectivemap}, note that we can recover the divisor class in $J(\mathbb{F}_q)[2]\setminus\{0\}$ as $\frac{1}{2}\divisor(x_0-\psi')$ for one of the three possible roots $\psi'$ of $F_\psi$ (and hence the map is at worst $3$-to-$1$).
\end{proof}

We now prove Theorem~\ref{thm:intro4}.

\begin{proof}[Proof of Theorem~\ref{thm:intro4}]
By Lemma~\ref{lemma:maintrick2},

\begin{align*}
|J(\mathbb{F}_q)[2]|-1\leq 3\sum_{\E\in \mathbf{E}}\left|\E(\mathbb{F}_q[t])_{\leq \frac{2g}{3}}\right|.
\end{align*} 

Taking $c=\lfloor\frac{2g}{3}\rfloor$ and $g_C=g$ in Theorem~\ref{generalMaronibound} (note the difference with the proof of Theorem~\ref{thm:intro3} in which the genus of the curve $C_a$ depends on the value of $a$), we find
\begin{align*}
\left|\E(\mathbb{F}_q[t])_{\leq \frac{2g}{3}}\right|\leq 2^{\frac{g+4}{3}+r_\E},
\end{align*} 
where $r_\E$ is the rank of $\E$ over $\mathbb{F}_q(t)$. Using Brumer's rank bound \eqref{eq:Brumer} and $|\mathbf{E}|\leq q^{g/3}-1$, we have that for some constant $\gamma$ depending only on $q$, 
\begin{align*}
|J(\mathbb{F}_q)[2]|\leq (2q)^{\frac{g}{3}+\gamma \frac{g}{\log g}}.
\end{align*} 
\end{proof}



\bibliographystyle{amsalpha}
\bibliography{biblioRP.bib}



\bigskip

\textsc{Jean Gillibert}, Universit\'e de Toulouse, Institut de Math{\'e}matiques de Toulouse, CNRS UMR 5219, 118 route de Narbonne, 31062 Toulouse Cedex 9, France.

\emph{E-mail address:} \texttt{jean.gillibert@math.univ-toulouse.fr}
\medskip

\textsc{Emmanuel Hallouin}, Universit\'e de Toulouse, Institut de Math{\'e}matiques de Toulouse, CNRS UMR 5219, 118 route de Narbonne, 31062 Toulouse Cedex 9, France.

\emph{E-mail address:} \texttt{hallouin@univ-tlse2.fr}
\medskip

\textsc{Aaron Levin}, Department of Mathematics, Michigan State University, 619 Red Cedar Road, East Lansing, MI 48824.

\emph{E-mail address:} \texttt{adlevin@math.msu.edu}


\end{document}